\documentclass[a4paper,11pt,twoside]{article}
\usepackage{latexsym,amsfonts,index,amssymb}
\usepackage[centertags]{amsmath}
\usepackage{amstext}
\usepackage{enumerate}
\usepackage{cases}
\usepackage{color}
\usepackage{enumitem}
\usepackage{bbold}
\usepackage{graphicx}
\usepackage{stmaryrd}

\newcommand{\hlslashslash}{%
  \raisebox{.4ex}{%
    \scalebox{.7}{%
      \rotatebox[origin=c]{40}{$\boldsymbol{-}$}%
    }%
  }%
}
\newcommand{\wlslash}{%
  {%
   \vphantom{\mathsf{w}}%
   \hspace*{-2.5pt}\ooalign{\kern.05em\smash{\hlslashslash}\hidewidth\cr$\hspace*{3.8pt}\mathsf{w}\hspace*{-1pt}$\cr}%
   \kern.05em
  }%
}
\newcommand{\hrslashslash}{%
  \raisebox{.4ex}{%
    \scalebox{.7}{%
      \rotatebox[origin=c]{-40}{$\boldsymbol{-}$}%
    }%
  }%
}
\newcommand{\wrslash}{%
  {%
   \vphantom{\mathsf{v}}%
  \hspace*{1.5pt} \ooalign{\kern.05em\smash{\hrslashslash}\hidewidth\cr$\hspace*{-1.2pt}\mathsf{w}\hspace*{-1pt}$\cr}%
   \kern.05em
  }%
}

\def\titlerunning#1{\gdef\titrun{#1}}
\makeatletter
\def\author#1{\gdef\autrun{\def\and{\unskip, }#1}\gdef\@author{#1}}
\def\address#1{{\def\and{\\\hspace*{18pt}}\renewcommand{\thefootnote}{}%
\footnote {#1}}%
\markboth{\autrun}{\titrun}} \makeatother
\def\email#1{e-mail: #1}
\def\subjclass#1{{\renewcommand{\thefootnote}{}%
\footnote{\emph{Mathematics Subject Classification (2010):} #1}}}
\def\keywords#1{\par\medskip
\noindent\textbf{Keywords.} #1}

\setenumerate[1]{label=$(\alph*)$,leftmargin=*}
\setenumerate[2]{label=$(\roman*)$,leftmargin=*}

\DeclareMathAlphabet{\mathpzc}{OT1}{pzc}{m}{it}
\definecolor{verde}{rgb}{0,.5,0}

\DeclareSymbolFont{bbold}{U}{bbold}{m}{n}
\DeclareSymbolFontAlphabet{\mathbbold}{bbold}
\newcommand\dd{\mathbbold{c}}
\newcommand\dl{\mbox{\textquoteright}\hspace*{-1pt}\mathbbold{c}}
\newcommand\dr{\mathbbold{c}\hspace*{-1.7pt}\mbox{\textquoteleft\hspace*{-0.8pt}}}

\def\crl#1{\ensuremath{|{#1}|_{\mathsf c}}}
\def\rk#1{\ensuremath{\mathop{\text{rank}}({#1})}}

\newcommand{\cev}[1]{\reflectbox{\ensuremath{\vec{\reflectbox{\ensuremath{#1}}}}}}
\newcommand{\bw}{\ensuremath{\mbox{\b{$\mathsf w$}}}}

\numberwithin{equation}{section}

\def\rk#1{\ensuremath{\mathop{\text{rank}}({#1})}}

\def\Z{{\mathbb Z}}
\newfont{\sss}{cmssi10 at 11pt}
\newfont{\bss}{cmssbx10 at 11pt}
\newfont{\tit}{cmitt10 at 11pt}

\newcommand{\pv}{pseudovariety}

\newcommand{\sm}{semigroup}

\newcommand{\Se}{{\bf S}}

\newcommand{\LI}{{\bf LI}}
\newcommand{\LG}{{\bf LG}}

\newcommand{\LV}{{\bf LV}}
\newcommand{\K}{{\bf K}}
\newcommand{\D}{{\bf D}}
\newcommand{\G}{{\bf G}}

\newcommand{\V}{{\bf V}}

\newcommand{\FV}{F_{\mathsf V}}

\newcommand {\pvv} {\mathbf V}

\newcommand{\kb}{\bar{\kappa}}
\newcommand{\kbt}{$\bar{\kappa}$-term}

\newcommand{\om}[2] {{\overline{\Omega}}_{#1}{\mathbf #2}}

\newcommand{\ome}[2] {{{\Omega}}_{#1}{\mathbf #2}}
\newcommand{\omek}[3] {{{\Omega}}^{#1}_{#2}{\mathbf #3}}

\newcommand{\mfrg}[3] {{#1}:{#2}\mathop{\hbox{\kern5pt$\circ$\kern-12pt\raise0.1pt\hbox
{$\longrightarrow$}}}{#3}}

\newtheorem{theorem}{Theorem}[section]
\newtheorem{proposition}[theorem]{Proposition}
\newtheorem{corollary}[theorem]{Corollary}

\newtheorem{lemma}[theorem]{Lemma}
\newtheorem{claim}{Claim}

\newtheorem{remark}[theorem]{Remark}
\newtheorem{example}[theorem]{Example}
\newenvironment{definition*}{\begin{trivlist}\item[\hskip
    \labelsep{\bf Definition\quad}]}%
  {\hfill\qed\end{trivlist}}
\newenvironment{notation*}{\begin{trivlist}\item[\hskip
    \labelsep{\bf Notation\quad}]}%
  {\end{trivlist}}

  \def\qed{{\unskip\nobreak\hfil\penalty50\hskip .001pt\hbox{}%
      \nobreak\hfil
      \vrule height 1.2ex width 1.1ex depth -.1ex
      \parfillskip=0pt\finalhyphendemerits=0\medbreak}}
\newenvironment{proof}{\begin{trivlist}\item[\hskip%
     \labelsep{\bf Proof.\quad}]}%
 {\hfill\qed\rm\end{trivlist}}

  {\hfill\qed\end{trivlist}}%

\begin{document}

\titlerunning{The $\kappa$-word  problem for ${\bf LG}$}

\title{\bf The word problem for $\kappa$-terms over the\\ pseudovariety of local groups}

\author{J. C. Costa %
  \and %
  C. Nogueira %
  \and %
  M. L. Teixeira%
}

\date{September 4, 2015}

\maketitle

\address{ %
  J. C. Costa \& M. L. Teixeira: %
  CMAT, Dep.\ Matem\'{a}tica e Aplica\c{c}\~{o}es, Universidade do Minho, Campus
  de Gualtar, 4710-057 Braga, Portugal; %
  \email{jcosta@math.uminho.pt, mlurdes@math.uminho.pt} %
  \and %
  C. Nogueira: %
   CMAT, Escola Superior de Tecnologia e Gest\~ao,
  Instituto Polit\'ecnico de Leiria,
  Campus 2, Morro do Lena, Alto Vieiro, 2411-901
   Leiria, Portugal; %
   \email{conceicao.veloso@ipleiria.pt} %
}

\subjclass{20M05, 20M07}

\begin{abstract}
   In this paper we study the $\kappa$-word problem for the pseudovariety ${\bf LG}$ of local
groups, where $\kappa$ is the canonical signature consisting of the multiplication and the pseudoinversion. We
solve this problem by transforming each arbitrary $\kappa$-term $\alpha$ into another one  called the canonical
form of $\alpha$ and by showing that different canonical forms have different interpretations over $\LG$. The
procedure of construction of these canonical forms consists in applying elementary changes determined by a
certain set $\Sigma$ of $\kappa$-identities. As a consequence, $\Sigma$ is a  basis of $\kappa$-identities for
the $\kappa$-variety generated by $\LG$.

  \keywords{Local group, pseudovariety, finite semigroup, implicit signature, word problem, $\kappa$-term, canonical form.}
\end{abstract}

\section{Introduction}

The notion  of a {\em pseudovariety} has played a key role in the
classification of finite semigroups. Recall that a pseudovariety of semigroups  is a class of finite semigroups
 closed under taking subsemigroups, homomorphic
images and finite direct products. The semidirect product operator on pseudovarieties of semigroups has received particular attention,
as it  allows to decompose complicated pseudovarieties into simpler ones, and
which in turn is central to the applications of semigroup theory in computer science.
Among the   most studied semidirect products of pseudovarieties are those of
the form $\V * \D $, where \V\ is any pseudovariety and \D\ is the pseudovariety of all finite semigroups whose idempotents are right
zeros~\cite{Straubing:1985,Tilson:1987, AlmeidaAzevedoTeixeira:02}.
 If \V\ is a pseudovariety of monoids, then \LV\ denotes the pseudovariety of all semigroups $S$
whose local submonoids are in \V\ (i.e., $eSe\in\V$ for all idempotents $e$ of $S$). In general, $\V *\D$ is a subpseudovariety of \LV\
but under certain conditions on the pseudovariety \V\ the equality holds~\cite{Straubing:1985, TherienWeiss:1985, Tilson:1987}. In particular, for the   pseudovariety \G\ of all finite groups,
 \LG\  is the  class of finite local groups
 and it is well-known that $\LG=\G*\D$~\cite{Stifler:1973}.

Many applications involve solving the membership
problem for specific pseudovarieties. A pseudovariety for which this
is possible is said to be \emph{decidable}. However, the semidirect
product does not preserve decidability
\cite{Auinguer&Steinberg:2003, Rhodes:1999}, and thus it is worth
investigating stronger properties of the factors under which
decidability of the semidirect product is guaranteed. This is the approach followed by
 Almeida and Steinberg that lead to the notion of \emph{tameness}~\cite{Almeida&Steinberg:2000a,Almeida&Steinberg:2000b}.

For a signature (or a type) $\sigma$ of algebras and  a class $\mathcal C$ of algebras of type $\sigma$ (i.e., $\sigma$-algebras),
 the {\em $\sigma$-word problem
for  $\mathcal C$}  consists in determining whether two given elements of the term algebra of type $\sigma$ (i.e., $\sigma$-terms) over an alphabet
have the same interpretation over every $\sigma$-algebra of $\mathcal C$.
In the context of the study of tameness of pseudovarieties of semigroups,
 it is necessary to study  the decidability of
 the $\sigma$-word problem over a pseudovariety \V , where $\sigma$ is a set of implicit operations on semigroups containing
 the multiplication, called an {\em implicit signature},
 since that  is one of the properties required for   \V\  to be tame.
For pseudovarieties of aperiodic semigroups it is common to use the signature $\omega$ consisting of the
multiplication and the $\omega$-power. For instance, the $\omega$-word problem is already solved for the
pseudovarieties ${\bf A}$ of all finite aperiodic semigroups~\cite{McCammond:2001,Zhiltsov:1999}, ${\bf J}$ of
${\cal J}$-trivial semigroups~\cite{Almeida:1990}, \LI\ of locally trivial semigroups~\cite{Almeida&Zeitoun:2003}, ${\bf R}$ of ${\cal R}$-trivial
semigroups~\cite{Almeida&Zeitoun:2007} and ${\bf LSl}$ of local semilattices~\cite{Costa:2001}. For
non-aperiodic cases, in which the $\omega$-power is not enough, it is common to use the  signature $\kappa$
consisting of the multiplication and the $(\omega-1)$-power, usually called the {\em canonical signature}. We
will use an extension of $\kappa$, denoted $\kb$, defined by $\kb=\{ab , a^{\omega +q}: q\in\mathbb Z\}$. It is
easy to realize that the $\kb $-word problem is equivalent to the $\kappa $-word problem.
 As examples of pseudovarieties for which the $\kappa$-word problem is solved, we cite the
pseudovarieties $\Se$ of all finite semigroups~\cite{Costa:2014} and  ${\bf CR}$ of completely regular
semigroups~\cite{Almeida&Trotter:2001}.

This paper  is a continuation of the work initiated in~\cite{Costa&Nogueira&Teixeira:2015}. In that paper, the
authors have shown that ${\bf LG}$ and $\Se$ verify exactly the same identities involving $\kb$-terms of rank 0
or 1, and have given a proof (alternative to that contained in~\cite{Costa:2014}) of the decidability of those
$\kb$-identities. As one recalls, the {\em rank} of a $\kb$-term is the maximum number of nested
$(\omega+q)$-powers in it. The present paper completes the proof of the decidability of the  $\kb$-word problem
(and, as a consequence, of the $\kappa$-word problem) over the pseudovariety \LG. We prove first that this
problem can be reduced to consider only identities involving $\kb$-terms with rank at most 2. Next, a canonical
form for rank 2 $\kb$-terms over \LG\ is defined, thus extending the notion of canonical  $\kb$-terms given
in~\cite{Costa&Nogueira&Teixeira:2015} for rank 0 and 1. Finally, for canonical $\kb$-terms $\alpha$ and
$\beta$, we show that the $\kb$-identity $\alpha=\beta$ holds over $\LG$  if and only if $\alpha$ and $\beta$
are the same $\kb$-term.
 Since it is shown that each $\kb$-term can be algorithmically transformed into a unique canonical form  with the same value over
${\bf LG}$, to test whether a $\kb$-identity $\alpha=\beta$ holds over $\LG$ it then suffices to verify if the
canonical forms of the $\kb$-terms $\alpha$ and $\beta$ are equal.

A fundamental tool in our work is that of $\mathbbold{q}$-\emph{root} of a  $\kb$-term $\alpha$ that is in rank
1 canonical form or rank 2 semi-canonical form.  The notion of a rank $i+1$ (with $i\geq 0$) semi-canonical form
was introduced by the first author in~\cite{Costa:2014} as a rank $i+1$ $\kb$-term $\alpha$ whose 2-expansion
(i.e., the $\kb$-term obtained from $\alpha$ by replacing the exponents $\omega+q$ of rank $i+1$ powers by 2) is
a rank $i$ canonical form (over $\Se$). For a certain parameter $\mathbbold{q}_{\alpha}$, which is a positive
integer and depends only on the $\kb$-term $\alpha$, and for any given
$\mathbbold{q}\geq\mathbbold{q}_{\alpha}$,  the $\mathbbold{q}$-\emph{root} of $\alpha$ is a well-determined
word ${\widetilde{\mathsf w}}_\mathbbold{q}({\alpha})$, over a certain alphabet $\mathsf{V}\cup
\mathsf{V}^{-1}$, which is reduced in the free group $\FV$ generated by $\mathsf{V}$. The main property is that,
if $\alpha$ and $\beta$ are $\kb$-terms
 as above and ${\widetilde{\mathsf w}}_\mathbbold{q}({\alpha})$ and ${\widetilde{\mathsf w}}_\mathbbold{q}({\beta})$ are the $\mathbbold{q}$-roots
  of $\alpha$ and $\beta$ with $\mathbbold{q}\geq \mbox{max}\{\mathbbold{q}_{\alpha},\mathbbold{q}_{\beta}\}$, then \LG\
  satisfies $\alpha=\beta$  if and only if ${\widetilde{\mathsf w}}_\mathbbold{q}({\alpha})={\widetilde{\mathsf w}_{\mathbbold q}}({\beta})$. Since
  an arbitrary rank 2 $\kb$-term can be algorithmically transformed into one in semi-canonical form (see~\cite{Costa:2014}),
   the above result provides an alternative criterion to decide the equality of $\kb$-identities over the class of all
    finite local groups. Moreover, each  word ${\widetilde{\mathsf w}}_\mathbbold{q}({\alpha})$ is obtained as the reduced form in
    the free group $\FV$ of another word ${\mathsf w}_\mathbbold{q}({\alpha})$, called the $\mathbbold{q}$-\emph{outline} of
    $\alpha$. The reduction process of an outline ${\mathsf w}_\mathbbold{q}({\alpha})$ into the root ${\widetilde{\mathsf w}}_\mathbbold{q}({\alpha})$
was fundamental to us in the definition of a canonical form for rank 2 $\kb$-terms over \LG\ since it served as
a guide to some of the simplifications that should be operated at the $\kb$-term level. Informally speaking, if
a $\kb$-term $\alpha$ may be transformed into another $\kb$-term $\beta$ and the outline ${\mathsf
w}_\mathbbold{q}({\beta})$ is ``closer'' than the outline ${\mathsf w}_\mathbbold{q}({\alpha})$ to their common
reduced form
 ${\widetilde{\mathsf w}}_\mathbbold{q}({\alpha})$($= {\widetilde{\mathsf w}}_\mathbbold{q}({\beta})$), then $\beta$ should be considered to be
``simpler'' than $\alpha$.  The notion of $\mathbbold{q}$-outline, introduced  here for $\kb$-terms over $\LG$, plays a similar role as a more general notion of superposition homomorphism for pseudowords that was used by Almeida and Azevedo~\cite{Almeida&Azevedo:1993} to provide a representation of the
free pro-$(\V*\D)$ semigroup over $A$ (see~\cite[Theorem 10.6.12]{Almeida:1992}).

 The paper is organized as follows. Section~\ref{section:preliminaries} recalls basic notions on finite semigroup theory and
 set the basic  notation and terminology for $\kb$-terms. In Section~\ref{section:cis}, we prove that the $\kb$-word problem
 over \LG\ can be reduced to $\kb$-terms with rank at most 2.  Section~\ref{section:algorithm_normal_form_rank2} describes the process
 of construction of a canonical form for rank 2 $\kb$-terms interpreted over \LG . The $\mathbbold{q}$-outline and the $\mathbbold{q}$-root of a $\kb$-term are introduced  in Section~\ref{section:characterizing_kterms}, and their fundamental properties are exhibited. Finally, in Section~\ref{section:main results}, we present and prove the main results of the paper.

\section{Preliminaries}\label{section:preliminaries}
 This section introduces briefly the most essential preliminaries, including some terminology and notation. We assume familiarity with
 basic results of the theory of
   pseudovarieties  and implicit operations. For further details and  general
background  see~\cite{Almeida:1992,  Rhodes&Steinberg:2009}. For the main definitions and basic results about
combinatorics on words, the reader is referred to~\cite{Lothaire:2002}.

\subsection{Words, pseudowords and $\kb$-terms}
In this paper, we consider a finite alphabet $A$ provided with a total order. The free semigroup (resp.\ the
free monoid) generated by $A$ is denoted by $A^+$ (resp.\ $A^*$). An element $w$ of $A^*$ is called a (finite)
word and its length is represented by $|w|$. The empty word is denoted by $\epsilon$ and its length is $0$. A
word is said to be {\em primitive} if it cannot be written in the form $u^n$ with $n>1$.  Two words $u$ and $v$
are {\em conjugate} if there are words $w_1,w_2\in A^*$ such that $u=w_1w_2$ and $v=w_2w_1$. A {\em Lyndon word}
is a primitive word which is minimal in its conjugacy class, for the lexicographic order that extends to $A^+$
the order on $A$.

 For a \pv\ $\pvv$,  a  {\em pro-}\V\  \sm\  is a compact semigroup which is residually in \V .
   We denote by  $\om{A}{\V}$ the  pro-\V\  \sm\  freely   generated by an alphabet $A$ and denote by $\ome{A}{\V}$ the free semigroup over \V\ generated by $A$, which is a subsemigroup  of $\om{A}{\V}$. The
   elements of  $\om{A}{\V}$, usually called {\em pseudowords} (over \V),  are naturally interpreted as
    {\em ($A$-ary) implicit operations} (operations that commute with homomorphisms between semigroups of \V ). A {\em pseudoidentity}
      is a formal equality between two pseudowords over the pseudovariety \Se\ of all finite semigroups.

 Given an element $s$ of a compact topological semigroup, the closed
subsemigroup generated by $s$ contains a unique idempotent, denoted
$s^{\omega}$. For $q\in\mathbb N$,   $s^{\omega+q} \ (=s^\omega s^q)$ belongs to the maximal closed subgroup containing
$s^{\omega}$, and its group inverse is denoted by $s^{\omega-q} (=(s^q)^{\omega-1})$.
  The following examples of  implicit operations play an important role in the next
 sections: the binary implicit operation {\em multiplication} interpreted as the \sm\
   multiplication on each  profinite  \sm , and, for each $q\in\mathbb Z$, the unary implicit
   operation  {\em $(\omega +q)$-power} which,
  for a profinite  \sm\ $S$, sends $s\in S$ to $s^{\omega +q}$.

  We denote by $\kb$ the implicit signature $\kb=\{ab , a^{\omega +q}: q\in\mathbb Z\}$.
  Every profinite semigroup has a natural structure of a $\kb$-semigroup,
via the interpretation of implicit operations as operations on profinite semigroups. The free $\kb$-algebra
generated by $A$ in the variety defined by the identity $x(yz)=(xy)z$ will be denoted by $T_A^{\kb}$ and its
elements will be called
 \emph{$\kb$-terms}. Sometimes we will omit the reference to the signature $\bar\kappa$ simply by referring to an element
of $T_A^{\kb}$ as a term.  Terms of the form $\pi^{\omega+q}$ will be called \emph{limit terms}, and $\pi$ and
$\omega+q$ will be called, respectively, its \emph{base} and its \emph{exponent}. For convenience, we allow the
empty term which is identified with the empty word $\epsilon$.

\subsection{Portions of a $\kb$-term}

 The {\em rank} of a $\kb$-term $\alpha\in T_A^{\kb}$ is the maximum number $\rk \alpha$  of nested infinite powers in it. So, the $\kb$-terms
  of rank $0$ are the words from $A^*$ and a $\kb$-term of rank $i+1$ is an
expression $\alpha$ of the form
$$\alpha=\rho_0\pi_1^{\omega+q_1}\rho_1\cdots \pi_n^{\omega+q_n}\rho_n,$$
where $n\geq 1$,  $\rk {\rho_j}\leq i$,  $\rk {\pi_\ell}= i$ and $q_\ell\in\mathbb Z$. We
call this form the \emph{rank configuration} of $\alpha$. The number $n$ is said to be the
$(i+1)$-\emph{length} of $\alpha$.  The subterms $\rho_0\pi_1^{\omega+q_1}$, $\pi_n^{\omega+q_n}\rho_n$ and
$\pi_j^{\omega+q_j}\rho_j \pi_{j+1}^{\omega+q_{j+1}}$ will be called, respectively, the \emph{initial portion},
the \emph{final portion} and the \emph{crucial portions} of $\alpha$. For a positive integer $p$, the \emph{$p$-expansion} of $\alpha$ is
the rank  $i$ $\kb$-term
$$\alpha^{(p)}=\rho_0\pi_1^p\rho_1\cdots \pi_n^p\rho_n.$$

Suppose that $i=0$, whence $\rk {\alpha}=1$. The $\omega$-terms $\rho_0\pi_1^{\omega}$, $\pi_n^{\omega} \rho_n$
and $\pi_j^{\omega}\rho_j \pi_{j+1}^{\omega}$ will be said to be, respectively, the \emph{initial
$\omega$-portion}, the \emph{final $\omega$-portion} and the \emph{crucial $\omega$-portions} of $\alpha$. In case $i=1$, so that $\rk {\alpha}=2$, the (rank 1) \emph{initial $\omega$-portion}, \emph{final $\omega$-portion} and \emph{crucial $\omega$-portions} of $\alpha$ are, respectively, the \emph{initial $\omega$-portion}, \emph{final $\omega$-portion} and \emph{crucial $\omega$-portions} of the 2-expansion $\alpha^{(2)}$ of $\alpha$.  For example, if
$\alpha= b(ab^\omega a)^{\omega-1}bc( c^{\omega-1}aa (bc)^{\omega-2})^{\omega-1}a^{\omega+1}$,
 then $bab^\omega$ and $a^{\omega}$  are the initial and the final $\omega$-portions, respectively, and $b^\omega aab^\omega$, $b^\omega abc c^{\omega}$,
 $c^{\omega} aa(bc)^{\omega}$, $(bc)^{\omega} c^{\omega}$ and $(bc)^{\omega} a^{\omega}$  are the crucial $\omega$-portions of $\alpha$.

\subsection{$\kb$-identities}
A  $\kb$-\textit{identity} over $A$ is a formal equality $u=v$ with $u,v\in T_A^{\kb}$. For a pseudovariety \V ,
we denote by $\omek{\kb}{A}{\V}$ the free $\kb$-\textit{semigroup} generated by $A$ in the variety of
$\kb$-semigroups generated by \V\ and notice that it is the $\kb$-subsemigroup of $\om{A}{\pvv}$ generated by
$A$. Elements of $\omek{\kb}{A}{\V}$ are called $\kb$-\textit{words} over \V . The unique ``evaluation''
homomorphism of $\kb$-semigroups $T_A^{\kb}\to\omek{\kb}AV$ that sends each letter $a\in A$ to itself is
denoted~$\varepsilon_{A,{\V}}^{\kb}$. Hence the  {\em $\kb$-word problem for \V } consists in determining
whether two given $\kb$-terms $\pi,\rho\in T_A^{\kb}$ satisfy the equality $\varepsilon_{A,\V}^{\kb}
(\pi)=\varepsilon_{A,\V}^{\kb}( \rho)$, i.e., whether $\pi$ and $\rho$ represent the same $\kb$-word of
$\omek{\kb} {A}{\pvv}$. If so, we write $\V\models \pi=\rho$, as usual. Note that the elements of $\kb$, when
viewed as pseudowords over $\V$, are elements of $\omek{\kappa}{A}{\V}$ since  the following identities hold
over every finite semigroup:  $x^{\omega+q}= {x}^{\omega-1}x^{q+1}$ and $x^{\omega-q}= {(x^q)}^{\omega-1}=
(x^{\omega-1})^q$, where $q\in\mathbb N_0$. So, informally speaking, we can say that  $\kb$ and $\kappa$ have
the same expressive power,
  in  the sense that $\omek{{\kb}}{A}{\V}$ is isomorphic to $\omek{\kappa}{A}{\V}$ and, consequently, the $\kb $-word problem
  is equivalent to the $\kappa $-word problem.

\subsection{Rewriting rules for $\kb$-terms over $\Se$}
 The following set $\Sigma_{\Se}$ of $\kb$-identities (where $p,q\in\Z$
and $n\in\mathbb N$)
\begin{numcases}{}
  (x^{\omega+p})^{\omega+q}=x^{\omega+pq},\label{eq:id_k-terms1}\\
 (x^{n})^{\omega+q}=x^{\omega+nq},\label{eq:id_k-terms2} \\
 x^{\omega+p} x^{\omega+q}=x^{\omega+p+q},\label{eq:id_k-terms3} \\
  x^{\omega+q}x^n =x^{\omega+q+n},\quad x^n  x^{\omega+q}=x^{\omega+q+n}, \label{eq:id_k-terms4}\\
(xy)^{\omega+q}x =x(yx)^{\omega+q},\label{eq:id_k-terms5}
\end{numcases}
holds in the pseudovariety $\Se$ of all finite semigroups.
Notice  that, using~\eqref{eq:id_k-terms3}--\eqref{eq:id_k-terms5}, it is easy to deduce the identities
\begin{equation}
\begin{split}
x^\omega(x^{\omega+p}y)^{\omega+q}=(x^{\omega+p}y)^{\omega+q},\quad\quad&(x^{\omega+p}y)^{\omega+q}x^\omega=(x^{\omega+p}yx^\omega)^{\omega+q},
\label{eq:xomega_fora-dentro}\\[1mm]
(yx^{\omega+p})^{\omega+q}x^\omega=(yx^{\omega+p})^{\omega+q},\quad\quad  & x^\omega(yx^{\omega+p})^{\omega+q}=(x^\omega yx^{\omega+p})^{\omega+q}.
\end{split}
 \end{equation}

Each $\kb$-\textit{identity}
$r=(u=v)$ can be seen as two rewriting rules $\vec r:u\rightarrow v$ and
$\cev r:v\rightarrow u$ for the transformation of  $\kb$-terms into other  $\kb$-terms. If we rewrite a $\kb$-term $\pi$  interpreting
 a $\kb$-identity~(2.$i$), with $i\in\{1,2,3,4\}$, as a rewriting rule from left to right and applying it to a subterm of $\pi$, we say
 that we make a~{(2.$i$)-contraction}. The transformations resulting from interpreting the $\kb$-identities as rewriting rules on the opposite direction
    are called  {\it expansions}. An application of the identity~\eqref{eq:id_k-terms5}
   from left to right or from right to left will be called a  {\it shift right} and a  {\it shift left}, respectively.

 We will talk about the rank of a transformation of $\kb$-terms using a $\kb$-identity $\alpha=\beta$
as the number ${\rm max}\{\rk \alpha,\rk \beta\}$. For example, if we rewrite
    $a b^{\omega+1} b (ca^{\omega+1})^{\omega-1}ca^{\omega+1}$ as $a b^{\omega+1} b (ca^{\omega+1})^{\omega}$,
     or as  $a b^{\omega+2}  (ca^{\omega+1})^{\omega-1}ca^{\omega+1}$,
    making right~\eqref{eq:id_k-terms4}-contractions,  we say that it was made a rank 2
    contraction in the first case, and a rank 1 contraction in the second one.

\subsection{Canonical forms for $\bar\kappa$-terms over $\Se$}\label{subsection:Canonical_forms_for_S}

The first author has shown in~\cite{Costa:2014} that the above set $\Sigma_{\Se}$ is enough to derive from an
arbitrary $\kb$-term $\alpha$ any other  $\beta$ such that $\Se\models \alpha=\beta$. In particular,
$\Sigma_{\Se}$ is used to reduce $\alpha$ to its  $\Se$ canonical form.

Here, we briefly recall the (recursive) definition of the canonical forms (over $\Se$).
The rank 0  \kbt s are all considered to be canonical forms. Assuming that the rank~$i$ canonical forms are defined, a rank~$i+1$ canonical form
  is a $\kb$-term $\alpha=\rho_0\pi_1^{\omega+q_1}\rho_1\cdots \pi_n^{\omega+q_n}\rho_n$ in rank configuration such that:
\begin{enumerate}
\item[(a)] the $2$-expansion $\alpha^{(2)}=\rho_0\pi_1^2\rho_1\cdots \pi_n^2\rho_n$ of $\alpha$ is a rank~$i$ canonical form;\vspace*{-2mm}
\item[(b)] each $\pi_j$ is a Lyndon term of rank~$i$;\vspace*{-2mm}
\item[(c)] no $\pi_j$ is a suffix of  $\rho_{j-1}$;\vspace*{-2mm}
\item[(d)] no $\pi_j$ is a prefix of some term $\rho_{j}\pi_{j+1}^\ell$ with $\ell\geq 0$.
\end{enumerate}

 A \emph{semi-canonical form} is a $\kb$-term that verifies condition~(a) above.
  The (semi-)canonical forms enjoy the following usefull property. A term $\alpha$ is in (semi-)canonical form if and only if every subterm of $\alpha$ is in (semi-)canonical form, if and only if the initial portion, the final portion  and all of the crucial portions of $\alpha$ are in (semi-)canonical form.

 The algorithm to determine the canonical form of any $\kb$-term is described recursively on the rank of the term and consists in two major steps. The first step reduces the given $\kb$-term to a semi-canonical form and the second step completes the calculation of the canonical form. As referred above, in this paper we will only need to use $\kb$-terms of rank at most 2. On the other hand, all rank 1 $\kb$-terms are semi-canonical forms and, for the construction of the rank 2 canonical forms over \LG,  we may already depart from a  semi-canonical rank 2 $\kb$-term.

\subsection{Local groups}\label{subsection:local_groups}
A local group  $S$ is a semigroup such that $eSe$ is a group for each idempotent $e$ of $S$. Equivalently, we
may say that $S$ is a local group if and only if $S$ has no idempotents or $S$ has a completely simple minimal
ideal containing all its idempotents~\cite[Proposition 2.1]{Costa&Nogueira&Teixeira:2015}. Groups, completely
simple and  nilpotent semigroups are examples of local groups. The following is a list of 
 some important pseudovarieties of local groups, defined by pseudoidentities according to Reiterman's theorem,  that we will use below:
 \vspace{-2mm}\begin{itemize}
  \item $\textbf{K}=[\! [x^{\omega}y=x^{\omega} ]\!]$ and $\textbf{D}=[\! [yx^{\omega}=x^{\omega} ]\!]$ are  the
    classes  of all finite semigroups whose idempotents are left zeros and right zeros, respectively;\vspace{-1mm}

     \item $\textbf{G}=[\! [x^{\omega}y=yx^{\omega}=y ]\!]$  is the   class of all finite groups;\vspace{-1mm}

     \item $\textbf{LI}=[\! [x^{\omega}yx^{\omega}=x^{\omega}]\!]$ and $\textbf{LG}=[\![(x^{\omega}yx^{\omega})^{\omega}=x^{\omega}]\!]$ are the
     classes of all finite locally trivial semigroups and local groups, respectively.
\end{itemize}

 Recall that \LI\ is the join of \K\ with \D . Therefore, a pseudoidentity $\alpha=\beta$ holds in \LI\ if and only if it holds
 in both \K\ and \D.  In particular, when $\alpha$ and $\beta$ are  canonical forms over $\Se$, it is easy to verify that $\alpha=\beta$ holds in \LI\ if and only if $\alpha$ and $\beta$ have the same initial and final $\omega$-portions. We also recall that \G\ and \LI\ are subpseudovarieties of \LG, but \LG\ is not the join of \G\ with \LI.
Hence, if a $\kb$-identity (in general, a pseudoidentity) $\alpha=\beta$ holds in \LG, then it holds in both \G\
and \LI\  but the converse implication is not valid. It is well known that if a pseudovariety $\V$ contains  \G\
or \LI, then it does not satisfy any non-trivial identity. So, in particular,
 we may identify the \LG-free semigroup $\ome{A}{\LG}$ with $A^+$.

In~\cite{Costa&Nogueira&Teixeira:2015} the authors defined a class
 of local groups denoted by $\mathcal{S}(G,L,{\mathsf f})$ in which
  $G$ is a  group, $L\subseteq A^+$ is a factorial language (i.e., a language that is closed under taking non-empty factors)
   and   ${\mathsf f}:L\cup \ddot{L}\rightarrow
G$  is a map, where $\ddot{L}$ is the subset of $A^+\setminus L$ formed by the words whose proper factors belong
to $L$. The underlying set of  $\mathcal{S}(G,L,{\mathsf f})$   is
 $L\cup (L^1\times G\times L^1)$, where  $L$ is the set of non-regular elements, $G$ is isomorphic to the (unique) maximal subgroup,
  $L^1\times G\times L^1$ is the underlying set of the minimal ideal, and ${\mathsf f}$
 serves to define the semigroup operation.
 Moreover, the minimal ideal is a Rees matrix semigroup
$\mathcal{M}[G;L^1,L^1;P]$, where  $P=\big(\hat{\mathsf f}(uv)\big)_{u,v\in L^1}$ and $\hat{\mathsf
f}:A^*\rightarrow G$ is defined as follows: $\hat{\mathsf f}(\epsilon)= 1_G$, $\hat{\mathsf f}(w)= {\mathsf
f}(w_0){\mathsf f}(\ddot w_1){\mathsf f}(w_1){\mathsf f}(\ddot w_2)\cdots  {\mathsf f}(\ddot w_m){\mathsf
f}(w_m)$ where ${sc}_L[w]=(w_0,\ddot w_1,w_1,\ldots,\ddot w_m,w_m)$, $w_0,\ldots ,w_m \in L$ and $\ddot
w_1,\ldots,\ddot w_m\in \ddot{L}$. See Section~2 of~\cite{Costa&Nogueira&Teixeira:2015} for more details on the
computation of $\ddot{L}$, and of the sequence of coordinates
 ${sc}_L[w]$ of a word $w$ determined by a factorial language $L$, and on the
definition of the multiplication on $\mathcal{S}(G,L,{\mathsf f})$. We have also constructed  a finite local group
$S_{\pi,\rho}$  of the form  $\mathcal{S}(G,L,{\mathsf f})$, associated to each pair $(\pi,\rho)$  of rank 1 canonical $\kb$-terms,
such that $\LG\models \pi=\rho$ if and only if $S_{\pi,\rho}\models \pi=\rho$.

\section{Some properties of $\kb$-terms over $\LG$}\label{section:cis}
In this section, we show some features of $\kb$-terms interpreted on finite local groups.  We prove in special
that the word problem for these terms reduces to consider terms of rank at most 2.

\subsection{Rewriting rules for $\kb$-terms over $\LG$}
Let us consider the set of $\kb$-identities $\Sigma=\Sigma_{\Se}\cup\{(x^{\omega}y
x^{\omega})^{\omega}=x^{\omega}\}$, the union of $\Sigma_{\Se}$ with the singular set containing the
 $\kb$-identity $(x^{\omega}y x^{\omega})^{\omega}=x^{\omega}$ which defines \LG. As one notices, the left side of this identity is a rank 2
 $\kb$-term while the term in the right side has rank 1. This is the key identity for the transformation of $\kb$-terms into other ones of rank at most 2.
    Notice that the authors proved
in~\cite{Costa&Nogueira&Teixeira:2015} that the rank 0 and rank 1 canonical $\kb$-terms over \Se\ are also canonical over
\LG, in the sense that \LG\ does not identify two distinct  canonical forms. In the sequel we will show that the
set $\Sigma$ is sufficient to reduce any $\kb$-term to its \LG\ canonical form, which is a $\kb$-term  of rank
at most 2. In particular, we will construct an algorithm to compute the \LG\ canonical form of a  rank 2 $\kb$-term.

 We  say that two $\kb$-terms $\alpha$ and $\beta$ are
 \emph{$\Sigma$-equivalent} (or, simply, \emph{equivalent}) when $\Sigma\vdash \alpha=\beta$, that is, when
 the $\kb$-identity $\alpha=\beta$ is a (syntactic) consequence of  $\Sigma$. Notice that, obviously, if $\alpha$ and $\beta$ are \emph{$\Sigma$-equivalent}, then
$\LG\models \pi=\rho$. Our goal is to prove that the converse implication also holds. We will do this 
by showing that each rank 2 term can be transformed into a
$\Sigma$-equivalent canonical form and by proving that,
 if two given $\LG$ canonical forms are equal over $\LG$ then
they are precisely the same $\kb$-term.

Let $\pi$ be a $\kb$-term of rank at least 1. Then $\pi$ is of the form $\pi=ux^{\omega+q}w$
for some integer $q$ and some $\kb$-terms $u$, $x$ and $w$ with $x$ non-empty.
By~\eqref{eq:id_k-terms3}, it follows that  $\pi$ may be transformed into $ux^{\omega}x^{\omega+q}w$. Therefore
$\pi$ is $\Sigma$-equivalent (it is $\Sigma_\Se$-equivalent to be more precise) to some $\kb$-term of the
form $ux^{\omega}v$, and we will often use this fact without further reference. In particular, using notably~\eqref{eq:xomega_fora-dentro} and the $\kb$-identity $(x^{\omega}y x^{\omega})^{\omega}=x^{\omega}$, we may derive
\begin{equation}\label{eq:pi_omega1_pi}
\pi^{\omega+1}=u(x^{\omega}vu)^{\omega}x^{\omega}v=u(x^{\omega}vux^{\omega})^{\omega}v=ux^{\omega}v=\pi.
\end{equation}
\subsection{Reduction to rank 2}
 Notice that the $\kb$-identities $(x^{\omega}y
x^{\omega})^{\omega}=x^{\omega}(y x^{\omega})^{\omega}=(x^{\omega}y)^{\omega} x^{\omega}$ are derived from
$\Sigma_{\Se}$ and that, for arbitrary integers $p$ and $q$, $(x^{\omega+p}y x^{\omega+q})^{\omega}=x^{\omega}$
is a consequence of $\Sigma$. It is useful to notice the following variations of this $\kb$-identity, which may
be deduced easily from this equation  using~\eqref{eq:xomega_fora-dentro},
\begin{equation}\label{eq:id_k-terms6b}
x^{\omega+p}  (yx^{\omega+q})^{\omega}= x^{\omega+p}=(x^{\omega+q}y)^{\omega}x^{\omega+p}.
\end{equation}

Now, from these ones we deduce a peculiar property of exponents that, in certain conditions and
 with a change of sign, may shift from  inside to outside of $(\omega-1)$-powers and vice-versa,
\begin{equation}\label{eq:inside-outside}
x^{\omega+p} (yx^{\omega+q})^{\omega-1} = x^{\omega+p-q}  (yx^{\omega})^{\omega-1},\quad (x^{\omega+q} y)^{\omega-1}x^{\omega+p} = (x^{\omega} y)^{\omega-1}x^{\omega+p-q}.
\end{equation}
Indeed, we deduce  the first identity as follows (the second one being proved by symmetry)
$$\begin{array}{rl}
x^{\omega+p} (yx^{\omega+q})^{\omega-1}\hspace*{-2mm}&=x^{\omega+p-q} (x^{\omega+q}y)^{\omega-1}x^{\omega+q} =x^{\omega+p-q}
(x^{\omega+q}y)^{\omega-1}x^{\omega+q}(yx^{\omega})^{\omega} \\
&=x^{\omega+p-q} (x^{\omega+q}y)^{\omega}x^{\omega}(yx^{\omega})^{\omega-1} =x^{\omega+p-q}  (yx^{\omega})^{\omega-1}.
\end{array}$$

We gather in the following proposition a few $\kb$-identities exhibiting cancelation properties that will be important
in the reduction process. The second identity is an improvement of~\eqref{eq:id_k-terms6b} and is more suitable
for certain applications.
\begin{proposition} The following $\kb$-identities are consequences of $\Sigma$, for all  $p,q,r,s\in\Z$,
\begin{alignat}{2}
x^{\omega+p}  y(z^{\omega+q}wx^{\omega+r}y)^{\omega-1} z^{\omega+s}&=
x^{\omega+p-r}  (z^{\omega}wx^{\omega})^{\omega-1} z^{\omega+s-q},\label{eq:id_shortening} \\
x^{\omega+p}  y(x^{\omega+q}y)^{\omega-1} x^{\omega+r}&= x^{\omega+p-q+r},\label{eq:id_elimination}\\
(x^{\omega+p}y)^{\omega-1}x^{\omega+q}(zx^{\omega+r})^{\omega-1}&=
(x^{\omega}zx^{\omega+p-q+r}yx^{\omega})^{\omega-1}.\label{eq:id_agglutination}
\end{alignat}
 \end{proposition}
 \begin{proof} The deduction of~\eqref{eq:id_shortening} can be made using $\Sigma_{\Se}$,~\eqref{eq:id_k-terms6b} and~\eqref{eq:inside-outside}  as
follows

$$\begin{array}{rl}
x^{\omega+p}  y(z^{\omega+q}wx^{\omega+r}y)^{\omega-1} z^{\omega+s}\hspace*{-2mm}&=x^{\omega+p}(z^{\omega+q}wx^{\omega+r})^{\omega}  y(z^{\omega+q}wx^{\omega+r}y)^{\omega-1} z^{\omega+s}\\
&=x^{\omega+p}(z^{\omega+q}wx^{\omega+r})^{\omega-1}  (z^{\omega+q}wx^{\omega+r}y)^{\omega} z^{\omega+s}\\
&=x^{\omega+p}(z^{\omega+q}wx^{\omega+r})^{\omega-1} z^{\omega+s}
\\&=x^{\omega+p-r}(z^{\omega}wx^{\omega})^{\omega-1} z^{\omega+s-q}.
\end{array}$$

The second identity is an immediate consequence of the first one. For~\eqref{eq:id_agglutination}, we prove
$(x^{\omega}y)^{\omega-1}x^{\omega+q}(zx^{\omega})^{\omega-1}=
 (x^{\omega}zx^{\omega-q}yx^{\omega})^{\omega-1}$ which is a simpler and, clearly,
 equivalent condition. Using~\eqref{eq:inside-outside} in the first identity below, we have
\begin{alignat*}{2}
(x^{\omega}y)^{\omega-1}x^{\omega+q}(zx^{\omega})^{\omega-1}\!
&=(x^{\omega}y)^{\omega-1}x^{\omega}(zx^{\omega-q})^{\omega-1}
&\quad & \\
    & = (x^{\omega}y)^{\omega-1}(x^{\omega}zx^{\omega-q})^{\omega-1}x^{\omega}&\quad & \\
   & =(x^{\omega}y)^{\omega-1}(x^{\omega}zx^{\omega-q})^{\omega-1}(x^{\omega}zx^{\omega-q}yx^{\omega})^{\omega}
   &\quad & \\
     & =(x^{\omega}y)^{\omega-1}(x^{\omega}zx^{\omega-q})^{\omega}yx^{\omega}(x^{\omega}zx^{\omega-q}yx^{\omega})^{\omega-1}    &\quad & \\
    &  = (x^{\omega}y)^{\omega-1}x^{\omega}yx^{\omega}(x^{\omega}zx^{\omega-q}yx^{\omega})^{\omega-1}  \\
     &=x^{\omega}(x^{\omega}zx^{\omega-q}yx^{\omega})^{\omega-1}\\
      &=(x^{\omega}zx^{\omega-q}yx^{\omega})^{\omega-1}.\end{alignat*}
   This proves the proposition.
 \end{proof}

It is also useful to notice the following properties.
\begin{corollary}\label{cor:id_shortening2}
Let $\tau$ and $\sigma$ be $\kb$-terms with rank at least 1.
\begin{enumerate}
  \item If  $\LI\models\tau=\sigma$, then   $\Sigma\vdash    \sigma(\tau\sigma)^{\omega-1}= \tau^{\omega-1}$.  \label{eq:id_shortening2}
   \item  If  $\K \models\tau=\sigma$, then $\Sigma\vdash \sigma^{\omega-1}   \tau^{\omega-1}= (\tau^2\sigma)^{\omega-1}\tau   $. \label{agglutinationK}
     \item  If  $\D \models\tau=\sigma$, then $\Sigma\vdash \sigma^{\omega-1}   \tau^{\omega-1}= \sigma (\tau\sigma^2)^{\omega-1}$. \label{agglutinationD}
\end{enumerate}
\end{corollary}
\begin{proof}
 Suppose that $\LI\models \tau=\sigma$. Then $\tau$ and $\sigma$ are $\Sigma_{\Se}$-equivalent, respectively, to $\kb$-terms  of the form
 $ux^\omega \tau' y^\omega v $ and $ux^\omega \sigma'y^\omega v$.  Therefore, using $\Sigma_\Se$ and~\eqref{eq:id_shortening}, one derives
     $$\sigma(\tau\sigma)^{\omega-1}=ux^\omega\sigma'(y^\omega vux^\omega\tau'y^\omega vux^\omega\sigma')^{\omega-1} y^\omega v=ux^\omega\tau'(y^\omega vux^\omega\tau'y^\omega vux^\omega\tau')^{\omega-1} y^\omega v=\tau^{\omega-1},$$
  thus showing~\ref{eq:id_shortening2}.

 Now suppose that $\K \models\tau=\sigma$. Then $\tau$ and $\sigma$ are $\Sigma_{\Se}$-equivalent to $\kb$-terms  of the form
 $ux^\omega \tau' $ and $ux^\omega \sigma'$, respectively. So, the deduction of~\ref{agglutinationK} can be done, using $\Sigma_\Se$  and~\eqref{eq:id_agglutination}, as follows
   \begin{alignat*}{2}\sigma^{\omega-1}\tau^{\omega-1} & =(ux^\omega \sigma')^{\omega-1}(ux^\omega \tau'ux^\omega \tau')^{\omega-1}ux^\omega \tau'\\
   & =  u(x^\omega \sigma'u)^{\omega-1}x^\omega( \tau'ux^\omega \tau'ux^\omega)^{\omega-1} \tau'\\
  &= u(x^\omega \tau'ux^\omega \tau'u  x^\omega \sigma'ux^\omega )^{\omega-1} \tau'\\
      &=  ( \tau^2   \sigma )^{\omega-1}\tau.
   \end{alignat*}
  The proof of~\ref{agglutinationD} can be  made analogously.
   \end{proof}

 We proceed by showing that,  for any $\kb$-term $\alpha$, 
 it is possible to effectively compute a   $\kb$-term with rank at most 2 that is $\Sigma$-equivalent to $\alpha$.  Therefore, the
$\kb$-word problem for $\LG$ consists simply in determining whether two given $\kb$-terms of rank at most
2 have the same interpretation over \LG .

\begin{proposition}\label{prop:rank2}
Let $\alpha$ be an arbitrary \kbt . It is possible to effectively compute a $\kb$-term $\alpha'$ such that $\alpha'$ is
 $\Sigma$-equivalent to $\alpha$ and $\rk {\alpha'}\le 2$.
\end{proposition}
\begin{proof} The proof is made by induction on the rank of $\alpha$. For $\rk {\alpha}\leq 2$, the result holds trivially. Let now $\rk {\alpha}=i+1$ with $i\geq 2$ and suppose, by induction hypothesis, that the proposition holds for $\kb$-terms of rank $i$. Let $\alpha=\rho_0\pi_1^{\omega+q_1}\rho_1\cdots \pi_n^{\omega+q_n}\rho_n$ be the rank configuration of $\alpha$. By means of expansions of rank $i+1$ of types~\eqref{eq:id_k-terms2} and~\eqref{eq:id_k-terms4}, if necessary, we may assume that each $q_j$ is equal to $-1$. For each $j\in\{1,\ldots,n\}$, we claim that the rank $i+1$ $\kb$-term $\pi_j^{\omega-1}$ may be reduced into a $\Sigma$-equivalent rank $i$ $\kb$-term $\pi'_j$. The proof of the claim is made by induction on the $i$-length $m$ of $\pi_j$. Suppose first that $m=1$. Then, the rank configuration of $\pi_j$ is of the form $\pi_j=w_0\sigma^{\omega+p}w_1$. Using~\eqref{eq:inside-outside} and~\eqref{eq:pi_omega1_pi} in particular, one deduces
$$\begin{array}{rl}
\pi_j^{\omega-1}\hspace*{-2mm}&=\pi_j^{\omega-1}\pi_j\pi_j^{\omega-1}\\
&=w_0(\sigma^{\omega+p}w_1w_0)^{\omega-1}\sigma^{\omega+p}(w_1w_0\sigma^{\omega+p})^{\omega-1}w_1\\
&=w_0(\sigma^{\omega+1}w_1w_0)^{\omega-1}\sigma^{\omega+2-p}(w_1w_0\sigma^{\omega+1})^{\omega-1}w_1\\
&=(w_0\sigma w_1)^{\omega-1}w_0\sigma^{\omega+2-p}w_1(w_0\sigma w_1)^{\omega-1}.
\end{array}$$
This last $\kb$-term has rank $i$ and, so, we take it to be $\pi'_j$. Suppose now that $m>1$ and, by the induction hypothesis, that the claim holds for $\kb$-terms of $i$-length $m-1$. Let $\pi_j=w_0\sigma_1^{\omega+p_1}w_1\sigma_2^{\omega+p_2}w_2$ with $\rk{\sigma_1}=\rk{\sigma_2}=i-1$ and $w_0$ and $w_2$ with rank at most $i-1$. In this case, we use~\eqref{eq:id_shortening},~\eqref{eq:inside-outside} and~\eqref{eq:pi_omega1_pi} to deduce
$$\begin{array}{rl}
\pi_j^{\omega-1}\hspace*{-2mm}&=\pi_j^{\omega-1}\pi_j\pi_j^{\omega-1}\\
&=w_0(\sigma_1^{\omega+p_1}w_1\sigma_2^{\omega+p_2}w_2w_0)^{\omega-1}\sigma_1^{\omega+p_1}w_1(\sigma_2^{\omega+p_2}w_2w_0\sigma_1^{\omega+p_1}w_1)^{\omega-1}\sigma_2^{\omega+p_2}w_2\\
&=w_0(\sigma_1^{\omega+p_1}w_1\sigma_2^{\omega+p_2}w_2w_0)^{\omega-1}\sigma_1^{\omega+p_1}(\sigma_2^{\omega+p_2}w_2w_0\sigma_1^{\omega+p_1})^{\omega-1}\sigma_2^{\omega+p_2}w_2\\
&=w_0(\sigma_1^{\omega+1}w_1\sigma_2^{\omega+p_2}w_2w_0)^{\omega-1}\sigma_1^{\omega+2-p_1}(\sigma_2^{\omega+1}w_2w_0\sigma_1^{\omega+1})^{\omega-1}\sigma_2^{\omega+1}w_2\\
&=(w_0\sigma_1w_1\sigma_2^{\omega+p_2}w_2)^{\omega-1}w_0\sigma_1^{\omega+2-p_1}\sigma_2w_2(w_0\sigma_1\sigma_2w_2)^{\omega-1}.
\end{array}$$
The $\kb$-terms $\delta_1=w_0\sigma_1w_1\sigma_2^{\omega+p_2}w_2$ and $\delta_2=w_0\sigma_1^{\omega+2-p_1}\sigma_2w_2(w_0\sigma_1\sigma_2w_2)^{\omega-1}$ are   rank $i$ and $\delta_1$ has $i$-length $m-1$. By induction hypothesis, $\delta_1^{\omega-1}$  is $\Sigma$-equivalent to some rank $i$ term $\delta'_1$. Therefore, $\pi_j^{\omega-1}$ is $\Sigma$-equivalent to the rank $i$ $\kb$-term $\pi'_j= \delta'_1\delta_2$, thus proving the claim.
It follows that the rank $i$  $\kb$-term $\alpha_1=\rho_0\pi'_1\rho_1\cdots \pi'_n\rho_n$ is $\Sigma$-equivalent to $\alpha$. To conclude the proof it suffices to use the induction hypothesis on $\alpha_1$.
\end{proof}
\section{Canonical forms for $\bar\kappa$-terms over $\LG$}\label{section:algorithm_normal_form_rank2}
In this section, we define canonical forms for $\bar\kappa$-terms over $\LG$ and show how to compute a canonical
form of any given $\bar\kappa$-term. The rank 0 and rank 1 canonical $\kb$-terms over \LG\ were already
introduced in~\cite{Costa&Nogueira&Teixeira:2015} as coinciding with the rank 0 and rank 1 canonical forms over
\Se, since \LG\ and $\Se$ satisfy the same $\kb$-identities of rank at most 1. So, by
Proposition~\ref{prop:rank2} above, in order to complete the definition of the canonical forms over $\LG$ it
remains to introduce the rank 2 canonical forms. Contrary to what happens with \Se , we do not introduce these
canonical forms by stating ahead the properties they must satisfy. We start by describing an algorithm of
reduction for rank 2 $\bar\kappa$-terms, consisting in the application of elementary changes determined by the
$\kb$-identities of $\Sigma$. A $\kb$-term $\alpha'$ emerging after the reduction process applied to a
$\bar\kappa$-term $\alpha$ is then called an \LG\ \emph{canonical form}. By
Theorem~\ref{theo:equal_canonical_forms} below it follows that the $\kb$-term $\alpha'$ is unique and so we call
it \emph{the \LG\ canonical form of $\alpha$}.

\subsection{Some preliminary explanations}
The method to compute the canonical form of any given $\kb$-term $\alpha$ of rank 2 over \LG\ will be formally
exposed in Section~\ref{subsection:rank2canonical_form_algorithm} below. To clarify that process it is useful to
identify previously some situations that can happen and to establish the changes that can be performed in those
circumstances.

Let us notice that we may begin by applying to $\alpha$ the first step of the $\Se$ canonical form reduction algorithm. If the term thus obtained  has rank 1, then the reduction algorithm for rank 1 $\kb$-terms defined  in~\cite{Costa&Nogueira&Teixeira:2015}  (which is the second step of the rank 1  algorithm for $\Se$)
  produces the canonical form of $\alpha$.
 So, throughout  this work we focus in the case  in which $\alpha$ is a rank 2 semi-canonical $\kb$-term. It is not helpful to considerer $\alpha$ a canonical form over $\Se$ since this condition would be immediately lost. On the contrary, the semi-canonical form status will be  preserved  during the process. Moreover, by means of expansions of rank 2 of types~\eqref{eq:id_k-terms2} and~\eqref{eq:id_k-terms4}, if necessary, we may also assume that each exponent of rank 2 limit terms is equal to $\omega-1$. Unlike our choice for the rank 1 canonical form, we will keep those exponents
  that way throughout all the process and, so, also in the rank 2 canonical form.  Hence, the rank configuration
  of $\alpha$ is of the form
\begin{equation}\label{kb-term_rank2_algo}
\alpha=\rho_0\pi_1^{\omega-1}\rho_1\cdots \pi_n^{\omega-1}\rho_n.
\end{equation}
From now on, unless otherwise stated, all rank 2 $\kb$-terms are of this form.

 A main objective of the algorithm is to decrease as much as possible the $2$-length of the $\kb$-term. There are two ways of doing that.

\begin{description}
\item[{\textbf{Eliminations.}}] The first way is to apply the $\kb$-identity~\eqref{eq:id_elimination}. If, for some $i\in\{1,\ldots,n\}$, the subterm $\rho_{i-1}\pi_i^{\omega-1}\rho_i$ of $\alpha$ can be transformed, using only rank 2 shifts, into a term of the form $w_0x^{\omega+p}  w_1(x^{\omega+q} w_1)^{\omega-1} x^{\omega+r}w_2$, then we replace the  subterm $\rho_{i-1}\pi_i^{\omega-1}\rho_i$ by $w_0x^{\omega+p-q+r}w_2$ and say that it was made a \emph{limit term elimination (on position $i$)} in $\alpha$. Notice that the $\kb$-term thus obtained is a semi-canonical form since its $\omega$-portions are $\omega$-portions of $\alpha$ and $\alpha$ is a semi-canonical $\kb$-term.

\item[{\textbf{Agglutinations.}}] The second one is to apply either identity~\eqref{eq:id_agglutination} or  Corollary~\ref{cor:id_shortening2} \ref{agglutinationK}--\ref{agglutinationD}.\linebreak Suppose first that $n> 1$ and that, for some $i\in\{1,\ldots,n-1\}$,
 the subterm $\delta_{i,i+1}=\rho_{i-1}\pi_{i}^{\omega-1}\rho_i\pi_{i+1}^{\omega-1}\rho_{i+1}$ of $\alpha$ can be transformed, using only rank 2 shifts, into a $\kb$-term
 of the form  $w_0(x^{\omega+p}w_1)^{\omega-1}x^{\omega+q}(w_2x^{\omega+r})^{\omega-1}w_3$. Then, by~\eqref{eq:id_agglutination}, the term
 $\delta_{i,i+1}$ is equivalent to $w_0(x^{\omega}w_2x^{\omega+p-q+r}w_1x^{\omega})^{\omega-1}w_3$. This term is not in semi-canonical form, but its easy to transformed it into the semi-canonical term   $\delta=w_0(x^\omega w_2x^{\omega+p-q+r}w_1)^{\omega-1}x^{\omega}w_3$.
  We identify two  other situations (symmetric to each other) in which is possible to decrease the $2$-length of  $\alpha$, this time by applying Corollary~\ref{cor:id_shortening2}.  Suppose first that $\delta_{i,i+1}$ is transformed, using rank 2 shifts, into a term
 $w_0(ux^{\omega+p}w_1)^{\omega-1}(ux^{\omega+q}w_2)^{\omega-1}w_3$ in which the rank 2 limit terms are adjacent and have the
 same initial $\omega$-portion $ux^{\omega}$.  Then, using Corollary~\ref{cor:id_shortening2}~\ref{agglutinationK},  we obtain  the term  $\delta=w_0(u x^{\omega+q}w_2ux^{\omega+q}w_2ux^{\omega+p}w_1)^{\omega-1}ux^{\omega+q}w_2w_3$. Suppose next that $\delta_{i,i+1}$ can be transformed by rank 2 shifts into a term of the form $w_0(w_1x^{\omega+p}u)^{\omega-1}(w_2x^{\omega+q}u)^{\omega-1}w_3$. In this situation, using Corollary~\ref{cor:id_shortening2}~\ref{agglutinationD}, one obtains   $\delta=w_0w_1x^{\omega+p}u(w_2x^{\omega+q}uw_1x^{\omega+p}uw_1x^{\omega+p}u)^{\omega-1}w_3$.

 So, in the three situations above, the term $\alpha'$
 obtained from $\alpha$ by replacing the subterm $\delta_{i,i+1}$ with $\delta$  is an equivalent semi-canonical form of $\alpha$ that has a 2-length one unity
 less.  In each of these cases, the process of transformation of $\alpha$ into $\alpha'$ is called a \emph{limit terms agglutination (on positions $i$ and $i+1$)} in $\alpha$.
\end{description}

 The second aim of our reduction algorithm is to ``shorten'' as much as possible the bases of rank 2 limit
terms. We want, in particular, to  decrease  as much as possible their $1$-length  and to transform all possible exponents into $\omega$. This objective of the algorithm will be attained using $\kb$-identity~\eqref{eq:id_shortening}.

 \smallskip
 \paragraph{\textbf{Shortenings.}} We identify three cases in which~\eqref{eq:id_shortening} can be used with an effective simplification of the $\kb$-term.
\begin{enumerate}[label=\arabic*)]
 \item The first case is  easy to treat. Assume that some subterm $\rho_{i-1}\pi_i^{\omega-1}$ of $\alpha$  is of the form $w\sigma(\tau\sigma)^{\omega-1}$ where $\sigma$ and $\tau$ are rank 1 terms  with the same value over the pseudovariety $\LI$.
     By Corollary~\ref{cor:id_shortening2}\ref{eq:id_shortening2}, we may replace  $\rho_{i-1}\pi_i^{\omega-1}$ by $w\tau^{\omega-1}$ in $\alpha$. Note that since $\rho_{i-1}\pi_i^{\omega-1}$ is a semi-canonical form,  $w\tau^{\omega-1}$ is a semi-canonical form too.  We call this transformation a \emph{limit term shortening of type 1 (on position $i$)} in $\alpha$.

 \item In the second
  case we begin by applying~\eqref{eq:id_shortening} (almost) directly but then it may be necessary to make some
 adjustments because the semi-canonical form might have been lost. Suppose that, for some $i\in\{1,\ldots,n\}$, the subterm $\rho_{i-1}\pi_i^{\omega-1}\rho_i$ of $\alpha$ can be transformed, using only rank 2 shifts, into a term of the form
  $w_0x^{\omega+p_1}  w_1(z^{\omega+q_1}w_2x^{\omega+p_2}w_1)^{\omega-1} z^{\omega+q_2}w_3$. Then, by~\eqref{eq:id_shortening},  $\rho_{i-1}\pi_i^{\omega-1}\rho_i$ is equivalent to $\delta=w_0x^{\omega+p}  (z^{\omega}w_2x^{\omega})^{\omega-1} z^{\omega+q}w_3$, with $p=p_1-p_2$ and $q=q_2-q_1$.

     The,  possibly new, crucial $\omega$-portion $\gamma=x^{\omega} z^{\omega}$ of $\delta$  may not be a canonical form and so
     $\delta$ may not be in semi-canonical form.   Notice that  either $x=z$ or      the canonical form of $x^\omega z^\omega$  is of the type $x^{\omega+k}u_{x,z} z^{\omega+\ell}$ with $u_{x,z}\in A^*$.

       Suppose that  $x=z$.  If $q= 0$ or $q\neq p=0$,  then   $\delta$ is respectively equivalent to the  semi-canonical term   $\delta'=w_0x^{\omega+p}(w_2x^{\omega})^{\omega-1} w_3$ and $\delta'=w_0(x^{\omega}w_2)^{\omega-1} x^{\omega+q}w_3$. Otherwise,
        let $u_{x,x}$ be the least letter of the alphabet $A$ distinct from the first letter of $x$. Notice that we may assume $A$ not a singular set since, otherwise, every $\kb$-term is equivalent to a rank 1 limit term with base the only letter of $A$, and the problem is trivial in that case.
        Then $x^{\omega+t_1}u_{x,x} x^{\omega+t_2}$ is a canonical form for any integers $t_1$ and $t_2$ and $\delta$ is equivalent to the semi-canonical term $\delta'=w_0x^{\omega+p} u_{x,x} (x^{\omega}w_2x^{\omega}u_{x,x})^{\omega-1} x^{\omega+q}w_3$ by~\eqref{eq:id_shortening}.  In all cases  we replace the  subterm $\rho_{i-1}\pi_i^{\omega-1}\rho_i$  by $\delta'$ in  $\alpha$.

        If $x\neq z$ and $x^{\omega+k}u_{x,z} z^{\omega+\ell}$ is the canonical form of $x^\omega z^\omega$, then $x^{\omega+t_1+k}u_{x,z} z^{\omega+t_2+\ell}$ is the canonical form of $x^{\omega+t_1} z^{\omega+t_2}$ for any integers $t_1$ and $t_2$. In this case,
         $\delta$ may be transformed into the semi-canonical term $\delta'=w_0x^{\omega+p}u_{x,z}(z^{\omega}w_2x^{\omega}u_{x,z})^{\omega-1}z^{\omega+q}w_3$  by~\eqref{eq:id_shortening}.

In resume,  the  subterm  $\rho_{i-1}\pi_i^{\omega-1}\rho_i$ of $\alpha$ may be replaced by the semi-canonical term:
  \begin{enumerate}
 \item $w_0x^{\omega+p}(w_2x^{\omega})^{\omega-1} w_3\ $ if $\gamma$ is not canonical, $x=z$ and $q=0$;\vspace*{-0.5mm}
 \item $w_0(x^{\omega}w_2)^{\omega-1} x^{\omega+q}w_3\ $ if $\gamma$ is not canonical, $x=z$ and $q\neq p=0$;\vspace*{-0.5mm}
 \item $w_0x^{\omega+p}u_{x,z}(z^{\omega}w_2x^{\omega}u_{x,z})^{\omega-1}z^{\omega+q}w_3\ $ otherwise.
 \end{enumerate}
  We finish this process by making a rank 2 shift right in $\delta'$ whenever possible. In all cases, we say that it was made a \emph{limit term shortening of type 2 (on position $i$)} in $\alpha$.

\item Finally, we consider the case in which some subterm $\rho_{i-1}\pi_i^{\omega-1}\rho_i$ of $\alpha$ is of the form \linebreak $\rho_{i-1}(w_1x^{\omega+q}w_2)^{\omega-1}w_1x^{\omega+q'}\rho'_i$ or  $\rho'_{i-1}x^{\omega+q'}w_1(w_2 x^{\omega+q}w_1)^{\omega-1}\rho_i$  with   $w_1\in A^*$.  In this case  $\rho_{i-1}\pi^{\omega-1}\rho_i$ is replaced by
 $\rho_{i-1}(w_1x^{\omega}w_2)^{\omega-1}w_1x^{\omega+q'-q}\rho'_i$  and $\rho'_{i-1}x^{\omega+q'-q}w_1(w_2 x^{\omega}w_1)^{\omega-1}\rho_i$, respectively, by application of~\eqref{eq:inside-outside}.  This transformation will be called a \emph{limit term shortening of type 3 (on position $i$)} in $\alpha$.
 \end{enumerate}

Notice that in  the three cases above $\rho_{i-1}\pi_i^{\omega-1}\rho_i$ is replaced by a $\Sigma$-equivalent semi-canonical term, whence with the same initial and final $\omega$-portions, and with the same 2-length (which is equal to one).
\subsection{The rank 2 canonical form reduction algorithm}\label{subsection:rank2canonical_form_algorithm}
We now describe the procedure to compute the \LG\ canonical form of any given $\kb$-term $\alpha$ of rank 2. As explained in the previous subsection, $\alpha$ can be taken in semi-canonical form and with rank configuration of the form
$$\alpha=\rho_0\pi_1^{\omega-1}\rho_1\cdots \pi_n^{\omega-1}\rho_n.$$

  The algorithm begins by applying all possible limit term eliminations and limit terms agglutinations. Next, one should apply limit term shortenings. However for some $\kb$-terms there may be a conflict in the application (or not) of shortenings in consecutive positions $i$ and $i+1$, giving rise to different canonical forms. When possible, in order to obtain uniqueness, we choose to apply always a shortening  in the smaller  position $i$ (see the $\kb$-term $\gamma$ in Example~\ref{example:rank2_canonical_form} below). For that, we make rank 2 shifts right at each position $i$, eventually preceded by rank 2 left expansions of type~\eqref{eq:id_k-terms4} at position $i+1$ when useful (this is the case only when $\rho_iu'$ is a prefix of $\pi_i^p$, for some $p\in\mathbb N$,  and $u'\in A^+$ is a prefix of $\pi_{i+1}$, since we have already made all possible 2-length reductions). We call this a rank 2 \emph{extended shift right}.

The steps of the algorithm are the following.
\begin{enumerate}[label=(\arabic*)]
 \vspace{-1mm}\item Apply all possible limit term eliminations and limit terms agglutinations. In case the term $\alpha_1$ thus
  obtained is rank 1, declare $\alpha'$ to be $\alpha_1$ and stop (since $\alpha_1$ is already in rank 1 canonical form).
  \vspace{-1mm}\item Apply all possible rank 2 extended shifts right.
 \vspace{-1mm}\item Apply all possible limit term shortenings.

\end{enumerate}

Notice that, if a $\kb$-term  produced by the above algorithm is rank 1, then it is the term $\alpha_1$
obtained after step (1) since steps  (2)  and (3) do not change the $2$-length of the $\kb$-term. On the
other hand, after the application of steps (2) or (3) there is no need to restart the algorithm by returning to step (1). Indeed,
if after a limit term shortening it would be possible to make a limit term elimination or a limit terms agglutination, then those operations would previously be possible in step (1) since any transformation  on step (3) on a subterm $\rho_{i-1}\pi_i^{\omega-1}\rho_i$ of a $\kb$-term $\alpha$  does not modify the initial and final $\omega$-portions of the subterm.  The same holds in step~(2), even if it is necessary to make an expansion on position $i+1$. In such case,  we want to identify a subterm  of the form  $\rho_{i-1}\pi_i^{\omega-1}\rho_iu'$ with the aim to make a shift and to replace it by a term of the form $\rho_{i-1}\rho'_{i}{\pi'_i}^{\omega-1}$, making sure that the first letter  of $\pi'_i$ is not the first letter of the subterm immediately after  ${\pi'_i}^{\omega-1}$. Note that $u'$ must be a finite word since, in case $u'$ would have rank 1, it would previously be possible  to apply in step (1) an agglutination on $\rho_{i-1}\pi_i^{\omega-1}\rho_i\pi_{i+1}^{\omega-1}$.

The above is clearly a terminating process, and a $\kb$-term $\alpha'$ it generates is called an $\LG$ \emph{canonical
form}. Since $\alpha'$ is unique (meaning that the process is confluent) by
Theorem~\ref{theo:equal_canonical_forms}, it is called \emph{the $\LG$ canonical form of $\alpha$}. It is easy to verify that the following conditions are equivalent for the $\kb$-term $\alpha$:  $\alpha$ is in $\LG$ canonical form;  the application of any of the above steps does not modify $\alpha$;   $\alpha'=\alpha$;  every subterm of $\alpha$ is in \LG\ canonical form.

\begin{example}\label{example:rank2_canonical_form}
Applying the above procedure to  rank $2$ semi-canonical forms  one gets  sequences of $\kb$-terms as in the  following examples.
\begin{enumerate}
  \item $\alpha=b(ab)^{\omega-5} cb (ab)^{\omega+2} c\Bigl(b(ab)^{\omega+2} c\Bigr)^{\omega-1}ac^{\omega-3}\Bigl(b^\omega  a^{\omega-1} c\Bigr)^{\omega-1}
b^\omega  a^{\omega+1} c \Bigl(b^{\omega-2} ac a^{\omega+4} c\Bigr)^{\omega-1}b^{\omega+1} $
\begin{alignat*}{4}
\alpha\rightarrow&\ b(ab)^{\omega-5} cac^{\omega-3}b^\omega\Bigl(  a^{\omega}cb^{\omega-2} ac
a^{\omega+2}cb^\omega\Bigr)^{\omega-1}a^{\omega}cb^{\omega+1} & & \quad\mbox{ (step (1))}  \\
\rightarrow&\ b(ab)^{\omega-5} cac^{\omega-3}b^\omega  a^{\omega}c\Bigl( b^{\omega-2} ac
a^{\omega+2}cb^\omega a^{\omega}c\Bigr)^{\omega-1} b^{\omega+1} & &\quad \mbox{ (step (2))}\\
\rightarrow &\ b(ab)^{\omega-5} cac^{\omega-3}\Bigl( b^{\omega} ac
a^{\omega+2}c\Bigr)^{\omega-1}b^{\omega+3}  & &\quad \mbox{ (step (3))}
\end{alignat*}
It follows that $\alpha'=b(ab)^{\omega-5} cac^{\omega-3}\Bigl( b^{\omega} ac
a^{\omega+2}c\Bigr)^{\omega-1}b^{\omega+3}$ is the \LG\ canonical form of $\alpha$.
\item   $\gamma= d^\omega b \Bigl(a d^{\omega-1}cd^{\omega+3}bad^\omega b\Bigr)^{\omega-1} \Bigl(ab (cd)^{\omega-2}a\Bigr)^{\omega-1}$
\begin{alignat*}{4}
\gamma \rightarrow &\ d^\omega b  a \Bigl( d^{\omega-1}cd^{\omega+3}bad^\omega ba \Bigr)^{\omega-1} b (cd)^{\omega-2}a\Bigl (a b (cd)^{\omega-2}a a b (cd)^{\omega-2}a\Bigr)^{\omega-1} & &\quad \mbox{(step (2))}\\
\rightarrow  &\; \Bigl( d^{\omega-1}cd^{\omega+3}ba\Bigr)^{\omega-1} b (cd)^{\omega}a\Bigl (a b (cd)^{\omega-2}a a b (cd)^{\omega}a\Bigr)^{\omega-1} & & \quad\mbox{(step (3))}
 \end{alignat*}
So, $\gamma'=\Bigl( d^{\omega-1}cd^{\omega+3}ba\Bigr)^{\omega-1} b (cd)^{\omega}a\Bigl (a b (cd)^{\omega-2}a a b (cd)^{\omega}a\Bigr)^{\omega-1}$ is the \LG\ canonical form of $\gamma$.
\end{enumerate}
\end{example}

We observe that, in the above example, $\alpha'$ is also a canonical form over \Se, but  $\gamma'$  is not.
 We could use, alternatively, for the \LG\ canonical form of any $\kb$-term $\alpha$ the
 canonical form $\alpha''$ of $\alpha'$ over \Se.  Our option is however to work with $\alpha'$ since
  it is somewhat simpler to prove its uniqueness in this form. In the rest of the article, when referring to a
  canonical form we will be referring to an \LG\ canonical form.

\section{Characterizing rank 2 $\kb$-terms with finite words}\label{section:characterizing_kterms}
In~\cite{Costa&Nogueira&Teixeira:2015}, the authors have shown that, for rank 1 canonical $\kb$-terms $\pi$ and $\rho$, the $\kb$-identity $\pi=\rho$ holds over $\LG$  only when $\pi$ and $\rho$ are the same $\kb$-term. This
 is done by associating to the pair $(\pi,\rho)$ a pair  $({\mathsf w}_\pi,{\mathsf w}_\rho)$, of finite words over a new
 alphabet ${\mathsf V}$, such that  $\pi=\rho$ if and only if  ${\mathsf w}_\pi={\mathsf w}_\rho$. Afterwards, a finite
 local group $S_{\pi,\rho}$ is   associated to the pair  $({\mathsf w}_\pi,{\mathsf w}_\rho)$ and it is used as a
 test-semigroup to verify whether the $\kappa$-identity $\pi=\rho$ holds over $\LG$.

In this section, we extend the above construction to rank 2 $\kb$-terms $\alpha$ in semi-canonical form. This
will be accomplished by associating to $\alpha$ and any large enough positive integer $\mathbbold{q}$ a finite
word over ${\mathsf V}\cup {\mathsf V}^{-1}$. It is denoted by ${\mathsf w}_\mathbbold{q}(\alpha)$ and is called
the \emph{$\mathbbold{q}$-outline of $\alpha$}. Its reduced form in the free group $\FV$ will be denoted by
$\widetilde{\mathsf w}_\mathbbold{q}(\alpha)$ and named the \emph{$\mathbbold{q}$-root of $\alpha$}. The
fundamental property is that, if $\alpha$ and $\beta$ are $\kb$-terms in rank 1 canonical form or in rank 2 semi-canonical form, then
$\LG\models \alpha=\beta$ if and only if $\widetilde{\mathsf w}_\mathbbold{q}(\alpha)=\widetilde{\mathsf
w}_\mathbbold{q}(\beta)$.

\subsection{Outlines and roots of $\kb$-terms}\label{subsection:Outlines_roots_kterms}
For any $\kb$-term $\alpha$, let $$\mathbbold{q}_\alpha=1+\mbox{max}\{|q|:\mbox{$\omega+q$ occurs in $\alpha$}\}$$
and fix a positive integer $\mathbbold{q}$ such that $\mathbbold{q}\geq  \mathbbold{q}_\alpha$. For every symbol
$t_{*}$ representing an integer, we will often use the notation $\mathbbold t_{*}$ to represent the integer
$\mathbbold{q}+t_{*}$.

We begin by recalling  the definition of an outline of a rank 1 canonical $\kb$-term $\alpha$, introduced (without a name) in~\cite{Costa&Nogueira&Teixeira:2015}. We will make minor adjustments on that notion and on the notations. Let
$\alpha=u_{0}x_{1}^{\omega+q_{1}} u_{1}\cdots x_{n}^{\omega+q_{n}}u_{n}$
 be the rank configuration of $\alpha$ and  notice that $\alpha$ is $\Sigma_\Se$-equivalent to the term
 $$(u_{0}x_{1}^{\omega})x_{1}^{\omega+q_{1}}(x_{1}^{\omega} u_{1}x_{2}^{\omega})x_{2}^{\omega+ q_{2}}\cdots x_{n-1}^{\omega+q_{n-1}}(x_{n-1}^{\omega} u_{n-1}x_{n}^{\omega}) x_{n}^{\omega+ q_{n}}(x_{n}^{\omega}u_{n}).$$
If  the subterms $u_{0}x_{1}^{\omega}$, $x_{n}^{\omega}u_{n}$,  $x_{i}^{\omega} u_{i}x_{i+1}^{\omega}$ and $x_j$
are regarded as symbols (that we name \emph{variables},  to distinguish them from the letters of the alphabet
$A$) of a new alphabet ${\mathsf V}$ and each $\omega+q_j$ is replaced by the positive integer $\mathbbold q_j=\mathbbold{q}+q_{j}$, then
the above term is transformed into a finite word over the alphabet ${\mathsf
V}$. Notice that those subterms are precisely the initial $\omega$-portion, the final $\omega$-portion, the
crucial $\omega$-portions and the bases of limit terms of $\alpha$. They will be represented by
$\mathsf{i}_{u_{0},x_{1}}$, $\mathsf{t}_{x_{n},u_{n}}$,  $\mathsf{c}_{x_{i}, u_{i},x_{i+1}}$ and
$\mathsf{b}_{x_j}$ and called respectively an \emph{initial}, a \emph{final}, a \emph{crucial} and
a \emph{base} variable of the alphabet ${\mathsf V}$. So, we associate to $\alpha$ the following word
over  ${\mathsf V}$
 $${\mathsf w}_\mathbbold{q}(\alpha)=\mathsf{i}_{u_{0},x_{1}}\mathsf{b}_{x_{1}}^{\mathbbold q_{1}}\mathsf{c}_{x_{1}, u_{1},x_{2}}
 \mathsf{b}_{x_{2}}^{\mathbbold q_{2}}\cdots \mathsf{b}_{x_{n-1}}^{\mathbbold q_{n-1}}\mathsf{c}_{x_{n-1}, u_{n-1},x_{n}}\mathsf{b}_{ x_{n}}^{\mathbbold q_{n}}\mathsf{t}_{x_{n},u_{n}},$$
called the \emph{$\mathbbold{q}$-outline of $\alpha$}. We denote
$\bw_\mathbbold{q}(\alpha)=\mathsf{b}_{x_{1}}^{\mathbbold q_{1}}\mathsf{c}_{x_{1}, u_{1},x_{2}}
 \mathsf{b}_{x_{2}}^{\mathbbold q_{2}}\cdots \mathsf{b}_{x_{n-1}}^{\mathbbold q_{n-1}}\mathsf{c}_{x_{n-1}, u_{n-1},x_{n}}\mathsf{b}_{ x_{n}}^{\mathbbold q_{n}}$, so that
 ${\mathsf w}_\mathbbold{q}(\alpha)=\mathsf{i}_{u_{0},x_{1}}\bw_\mathbbold{q}(\alpha)\mathsf{t}_{x_{n},u_{n}}$. We remark that the initial and final variables were not used
in~\cite{Costa&Nogueira&Teixeira:2015}. The initial and final $\omega$-portions of the $\kb$-term were taken
into account by the introduction of two other (not so standard) variables. These two approaches are perfectly
homologous but the (minor) changes introduced here seem to be more natural.

The $\mathbbold{q}$-outline ${\mathsf w}_\mathbbold{q}(\alpha)$, of any rank 2 $\kb$-term $\alpha$ in
semi-canonical form, can be obtained by the application of
the two following recursive steps.
\begin{enumerate}[label=\arabic*)]
\item Consider $\alpha=\pi^{\omega-1}$, with $\pi=u_{0}x_{1}^{\omega+q_{1}} u_{1}\cdots x_{n}^{\omega+q_{n}}u_{n}$. Notice that, for every positive integer $k$, the $k$-expansion $\alpha^{(k)}\, (=\pi^k)$ is a
 canonical form by hypothesis and that the initial and final $\omega$-portions, $u_{0}x_{1}^{\omega}$ and $x_{n}^{\omega}u_{n}$, of $\pi$ are the initial and final $\omega$-portions of $\alpha$  and of $\alpha^{(k)}$ and
$${\mathsf w}_\mathbbold{q}(\alpha^{(k)})=\mathsf{i}_{u_{0},x_{1}} ( \mathsf{b}_{x_{1}}^{\mathbbold q_{1}}  \mathsf{c}_{x_{1}, u_{1},x_{2}}  \cdots       \mathsf{b}_{x_{n}}^{\mathbbold q_{n}}\mathsf{c}_{x_{n}, u_{n}u_{0},x_{1}})^{k-1}   \mathsf{b}_{x_{1}}^{\mathbbold q_{1}}  \mathsf{c}_{x_{1}, u_{1},x_{2}}    \cdots      \mathsf{b}_{x_{n}}^{\mathbbold q_{n}}     \mathsf{t}_{x_{n},u_{n}}.$$
Furthermore, in the free group $\FV$,
$${\mathsf w}_\mathbbold{q}(\alpha^{(k)})=\mathsf{i}_{u_{0},x_{1}} ( \mathsf{b}_{x_{1}}^{\mathbbold q_{1}}  \mathsf{c}_{x_{1}, u_{1},x_{2}}  \mathsf{b}_{x_{2}}^{\mathbbold q_{2}}   \cdots       \mathsf{c}_{x_{n-1}, u_{n-1},x_{n}}\mathsf{b}_{x_{n}}^{\mathbbold q_{n}}\mathsf{c}_{x_{n}, u_{n}u_{0},x_{1}})^{k}  \mathsf{c}_{x_{n}, u_{n}u_0,x_{1}}^{-1}     \mathsf{t}_{x_{n},u_{n}}.$$
Each finite group $G$ verifies $g^\ell=1_G$ for some positive integer $\ell>2$. Therefore, over $G$,
\begin{alignat*}{2} {\mathsf w}_\mathbbold{q}(\alpha^{(\ell-1)})& = \mathsf{i}_{u_{0},x_{1}}(\mathsf{b}_{x_{1}}^{\mathbbold q_{1}}\mathsf{c}_{x_{1}, u_{1},x_{2}} \cdots \mathsf{b}_{ x_{n}}^{\mathbbold q_{n}}\mathsf{c}_{x_{n},u_{n}u_{0},x_{1}}
)^{\ell-1}\;\!\mathsf{c}_{x_{n}, u_{n}u_0,x_{1}}^{-1}\mathsf{t}_{x_{n},u_{n}}\\[1mm]
&=\mathsf{i}_{u_{0},x_{1}}(\mathsf{b}_{x_{1}}^{\mathbbold q_{1}}\mathsf{c}_{x_{1}, u_{1},x_{2}}
 \cdots \mathsf{b}_{ x_{n}}^{\mathbbold q_{n}}\mathsf{c}_{x_{n},u_{n}u_0,x_1}
)^{-1}\;\!\mathsf{c}_{x_{n}, u_{n}u_0,x_{1}}^{-1}\mathsf{t}_{x_{n},u_{n}}\\[1mm]
&=\mathsf{i}_{u_{0},x_{1}} \mathsf{c}_{x_{n}, u_{n}u_{0},x_{1}}^{-1} \mathsf{b}_{x_{n}}^{-\mathbbold q_{n}}
 \cdots \mathsf{c}_{x_{1}, u_{1},x_{2}}^{-1}\mathsf{b}_{x_{1}}^{-\mathbbold q_{1}}\mathsf{c}_{x_{n}, u_{n}u_{0},x_{1}}^{-1}\mathsf{t}_{x_{n},u_{n}}.
\end{alignat*}
 In this case, we define the
\emph{$\mathbbold{q}$-outline} of $\alpha$ as the following word over the alphabet $\mathsf{V}\cup
\mathsf{V}^{-1}$,
$${\mathsf w}_\mathbbold{q}(\alpha)=\mathsf{i}_{u_{0},x_{1}} \mathsf{c}_{x_{n}, u_{n}u_{0},x_{1}}^{-1} \mathsf{b}_{x_{n}}^{-\mathbbold q_{n}}
 \mathsf{c}_{x_{n-1}, u_{n-1},x_{n}}^{-1}\mathsf{b}_{x_{n-1}}^{-\mathbbold q_{n-1}}\cdots \mathsf{b}_{x_{2}}^{-\mathbbold q_{2}}\mathsf{c}_{x_{1}, u_{1},x_{2}}^{-1}
 \mathsf{b}_{x_{1}}^{-\mathbbold q_{1}}\mathsf{c}_{x_{n}, u_{n}u_{0},x_{1}}^{-1} \mathsf{t}_{x_{n},u_{n}}.$$
Denoting $\bw_\mathbbold{q}(\alpha)=\mathsf{c}_{x_{n}, u_{n}u_{0},x_{1}}^{-1} \mathsf{b}_{x_{n}}^{-\mathbbold q_{n}}
 \mathsf{c}_{x_{n-1}, u_{n-1},x_{n}}^{-1}\cdots \mathsf{b}_{x_{2}}^{-\mathbbold q_{2}}\mathsf{c}_{x_{1}, u_{1},x_{2}}^{-1}
 \mathsf{b}_{x_{1}}^{-\mathbbold q_{1}}\mathsf{c}_{x_{n}, u_{n}u_{0},x_{1}}^{-1}$,   ${\mathsf w}_\mathbbold{q}(\alpha)$  may be written  as ${\mathsf w}_\mathbbold{q}(\alpha)=\mathsf{i}_{u_{0},x_{1}} \bw_\mathbbold{q}(\alpha)\mathsf{t}_{x_{n},u_{n}}$ also in this case.

\item Suppose that $\alpha=\alpha_1\alpha_2$ and notice that, as observed in Section~\ref{subsection:Canonical_forms_for_S},
 each subterm $\alpha_j$ is a semi-canonical form (and it is a canonical form when $\rk{\alpha_j}\leq 1$). If $\alpha_j$ is rank 1 or rank 2, then we assume  ${\mathsf w}_\mathbbold{q}(\alpha_j)$
   already defined and of the form
  ${\mathsf w}_\mathbbold{q}(\alpha_j)=\mathsf{i}_{u_{j},x_{j}}\bw_\mathbbold{q}(\alpha_j)\mathsf{t}_{y_{j},v_{j}}$.

    If $\alpha_1$ is rank 0, then we let ${\mathsf w}_\mathbbold{q}(\alpha)$ be the word
    $\mathsf{i}_{\alpha_1u_{2},x_{2}}\bw_\mathbbold{q}(\alpha_2)\mathsf{t}_{y_{2},v_{2}}$. Symmetrically, if $\alpha_2$
    is rank 0, then we take ${\mathsf w}_\mathbbold{q}(\alpha)=\mathsf{i}_{u_{1},x_{1}}\bw_\mathbbold{q}(\alpha_1)
    \mathsf{t}_{y_{1},v_{1}\alpha_2}$. Finally, for $\rk{\alpha_j}\in\{1,2\}$,  let
    ${\mathsf w}_\mathbbold{q}(\alpha)=\mathsf{i}_{u_{1},x_{1}}\bw_\mathbbold{q}(\alpha_1)
    \mathsf{c}_{y_{1},v_{1}u_{2},x_{2}}\bw_\mathbbold{q}(\alpha_2)\mathsf{t}_{y_{2},v_{2}}$. In this case,
    the crucial variable $\mathsf{c}_{y_{1},v_{1}u_{2},x_{2}}$ will be also denoted by
    $\mathsf{c}(\alpha_1,\alpha_2)$, whence ${\mathsf w}_\mathbbold{q}(\alpha)=\mathsf{i}_{u_{1},x_{1}}\bw_\mathbbold{q}(\alpha_1)
    \mathsf{c}(\alpha_1,\alpha_2)\bw_\mathbbold{q}(\alpha_2)\mathsf{t}_{y_{2},v_{2}}$.
\end{enumerate}

Let $\alpha$ be any rank 1 or rank 2 $\kb$-term  as above and let $ux^{\omega}$ and $y^{\omega}v$ be,
respectively, the initial and the final $\omega$-portions of $\alpha$. The variables $\mathsf{i}_{u,x}$ and
$\mathsf{t}_{y,v}$ will be also denoted respectively by $\mathsf{i}(\alpha)$ and $\mathsf{t}(\alpha)$. Then, by
the above definition, it is clear that  ${\mathsf w}_\mathbbold{q}(\alpha)$  may be written as
 \begin{equation}\label{eq:outline_initial_final_decomposition}
{\mathsf w}_\mathbbold{q}(\alpha)=\mathsf{i}(\alpha) \bw_\mathbbold{q}(\alpha)\mathsf{t}(\alpha)
\end{equation}
 for some word $\bw_\mathbbold{q}(\alpha)$. Moreover each of $\mathsf{i}(\alpha)$ and $\mathsf{t}(\alpha)$ has exactly one occurrence in the word ${\mathsf w}_\mathbbold{q}(\alpha)$.
 Now, let $\widetilde{\mathsf
w}_\mathbbold{q}(\alpha)$ be the reduced form of ${\mathsf w}_\mathbbold{q}(\alpha)$ in the free group $\FV$
generated by ${\mathsf V}$. The word $\widetilde{\mathsf w}_\mathbbold{q}(\alpha)$ will be called the
$\mathbbold{q}$-\emph{root} of $\alpha$. By~\eqref{eq:outline_initial_final_decomposition},
\begin{equation}\label{eq:root_initial_final_decomposition}
\widetilde{\mathsf w}_\mathbbold{q}(\alpha)=\mathsf{i}(\alpha)
\widetilde\bw_\mathbbold{q}(\alpha)\mathsf{t}(\alpha)
\end{equation}
where $\widetilde\bw_\mathbbold{q}(\alpha)$ is the reduced form of $\bw_\mathbbold{q}(\alpha)$ in $\FV$. In
particular, when
 $\rk \alpha=1$, the outline ${\mathsf w}_\mathbbold{q}(\alpha)$ is a word of ${\mathsf V}^+$ and, so, $\widetilde{\mathsf
w}_\mathbbold{q}(\alpha)={\mathsf w}_\mathbbold{q}(\alpha)$.

\begin{example}\label{example:computation_outline_root}
Consider the rank $2$ semi-canonical form    $\alpha$
 of
Example~\ref{example:rank2_canonical_form}. Hence, the $\mathbbold{q}$-outline and the $\mathbbold{q}$-root
of $\alpha$ are the following
$$\begin{array}{rl}
 {\mathsf w}_\mathbbold{q}(\alpha)\hspace*{-2mm}& =\mathsf{i}_{b,ab}\mathsf{b}_{ab}^{\mathbbold{q}-5} \mathsf{c}_{ab,cb,ab}
 \mathsf{b}_{ab}^{\mathbbold{q}+2}\mathsf{c}_{ab, cb,ab}\mathsf{c}_{ab,cb,ab}^{-1}\mathsf{b}_{ab}^{-(\mathbbold{q}+2)}\mathsf{c}_{ab,cb,ab}^{-1}
 \mathsf{c}_{ab,ca,c}\mathsf{b}_{c}^{\mathbbold{q}-3}\mathsf{c}_{c,\epsilon,b}\mathsf{c}_{a,c,b}^{-1}\mathsf{b}_{a}^{-(\mathbbold{q}-1)}\\[1mm]
& \hspace*{4mm} \mathsf{c}_{b,\epsilon,a}^{-1}\mathsf{b}_{b}^{-\mathbbold{q}}\mathsf{c}_{a,c,b}^{-1}
\mathsf{c}_{a,c,b}\mathsf{b}_{b}^{\mathbbold{q}}\mathsf{c}_{b,\epsilon,a}\mathsf{b}_{a}^{\mathbbold{q}+1}\mathsf{c}_{a,c,b}\mathsf{c}_{a,c,b}^{-1}\mathsf{b}_{a}^{-(\mathbbold{q}+4)}\mathsf{c}_{b,ac,a}^{-1}
\mathsf{b}_{b}^{-(\mathbbold{q}-2)}\mathsf{c}_{a,c,b}^{-1} \mathsf{c}_{a,c,b}\mathsf{b}_{b}^{\mathbbold{q}+1}\mathsf{t}_{b,\epsilon}\\[2mm]

\widetilde{\mathsf w}_\mathbbold{q}(\alpha)\hspace*{-2mm}&
=\mathsf{i}_{b,ab}\mathsf{b}_{ab}^{\mathbbold{q}-5}\mathsf{c}_{ab,ca,c}\mathsf{b}_{c}^{\mathbbold{q}-3}\mathsf{c}_{c,\epsilon,b}\mathsf{c}_{a,c,b}^{-1}
\mathsf{b}_{a}^{-(\mathbbold{q}+2)}\mathsf{c}_{b,ac,a}^{-1}\mathsf{b}_{b}^{3}\mathsf{t}_{b,\epsilon}.
\end{array}$$
The \LG\ canonical form of $\alpha$ is $\alpha'=b(ab)^{\omega-5} cac^{\omega-3}\Bigl( b^{\omega} ac a^{\omega+2}c\Bigr)^{\omega-1}b^{\omega+3}$ and, so,
$$\begin{array}{rl}
 {\mathsf w}_\mathbbold{q}(\alpha')\hspace*{-2mm}& =\mathsf{i}_{b,ab}\mathsf{b}_{ab}^{\mathbbold{q}-5}\mathsf{c}_{ab,ca,c}
 \mathsf{b}_{c}^{\mathbbold{q}-3}\mathsf{c}_{c,\epsilon,b}\mathsf{c}_{a,c,b}^{-1}\mathsf{b}_{a}^{-(\mathbbold{q}+2)}\mathsf{c}_{b,ac,a}^{-1}
 \mathsf{b}_{b}^{-\mathbbold{q}}\mathsf{c}_{a,c,b}^{-1}\mathsf{c}_{a,c,b}\mathsf{b}_{b}^{\mathbbold{q}+3}\mathsf{t}_{b,\epsilon}\\[2mm]

\widetilde{\mathsf w}_\mathbbold{q}(\alpha')\hspace*{-2mm}& =\widetilde{\mathsf w}_\mathbbold{q}(\alpha).
\end{array}$$
\end{example}

Notice that an outline is a way to encode a term. Indeed,  it is obvious that for  terms $\alpha$ and $\beta$,
${\mathsf w}_\mathbbold{q}(\alpha)={\mathsf w}_\mathbbold{q}(\beta)$   for all $\mathbbold{q}\geq
\mbox{max}\{\mathbbold{q}_{\alpha},\mathbbold{q}_{\beta}\}$    is equivalent to ${\mathsf
w}_\mathbbold{q}(\alpha)={\mathsf w}_\mathbbold{q}(\beta)$ for some
$\mathbbold{q}\geq\mbox{max}\{\mathbbold{q}_{\alpha},\mathbbold{q}_{\beta}\}$. This condition   implies that,
either $\alpha$ is $\beta$, or $\alpha$ can be obtained from $\beta$ by applying a finite number of rank 2
shifts of the form $x(wx)^{\omega-1}=(xw)^{\omega-1}x$ with $x\in A^+$.  Moreover, although the
$\mathbbold{q}$-outline and the $\mathbbold{q}$-root of a term depend on the given integer $\mathbbold{q}$,  the
truthfulness of an identity of the kind $\widetilde{\mathsf w}_\mathbbold{q}(\alpha)=\widetilde{\mathsf
w}_\mathbbold{q}(\beta)$ is independent of the value chosen for $\mathbbold q$  provided  that
$\mathbbold{q}\geq\mbox{max}\{\mathbbold{q}_{\alpha},\mathbbold{q}_{\beta}\}$ since, as illustrated  in
Example~\ref{example:computation_outline_root}, the cancelations performed in the reduction process to compute
the $\mathbbold{q}$-root of a term do not depend on a specific value of $\mathbbold q$.  When $\alpha$ and
$\beta$ are canonical forms, we can be more precise. In this case, $\alpha$ and $\beta$ are both irreducible for
rank 2 extended shifts right and, so, $\alpha=\beta$ if and only if  ${\mathsf w}_\mathbbold{q}(\alpha)={\mathsf
w}_\mathbbold{q}(\beta)$  for all/some
$\mathbbold{q}\geq\mbox{max}\{\mathbbold{q}_{\alpha},\mathbbold{q}_{\beta}\}$.
\subsection{A necessary condition for the identity of two $\kb$-terms over \LG}\label{subsection:nec_condition_for_kterms_over_LG}
In this section we show that a necessary condition for the equality of two $\kb$-terms over \LG\ is the equality
of their roots.  The proof that this  is also sufficient is left to the next section.
\begin{proposition}\label{prop:nec_condit_eq_kterms}
Let $\alpha$ and $\beta$ be  \kbt s  in rank $1$ canonical form or in rank $2$ semi-canonical form and let $\mathbbold{q}\geq {\rm max}\{\mathbbold{q}_{\alpha},\mathbbold{q}_{\beta}\}$.
 If  $\LG\models\alpha=\beta$, then $\widetilde{\mathsf w}_\mathbbold{q}(\alpha)=\widetilde{\mathsf w}_\mathbbold{q}(\beta)$.
\end{proposition}
\begin{proof}
Assume that $\LG\models\alpha=\beta$. Then $\LI\models\alpha=\beta$, which means, by~\eqref{eq:root_initial_final_decomposition}, that the $\mathbbold{q}$-roots $\widetilde{\mathsf w}_\mathbbold{q}(\alpha)$ and $\widetilde{\mathsf w}_\mathbbold{q}(\beta)$ of $\alpha$ and $\beta$ have
the same initial and final variables, say $\mathsf{i}_{u,x}$ and $\mathsf{t}_{y,v}$ respectively. Suppose, by way of contradiction, that
$\widetilde{\mathsf w}_\mathbbold{q}(\alpha)\neq\widetilde{\mathsf w}_\mathbbold{q}(\beta)$. The case in which
$\alpha$ and $\beta$ are both rank 1 $\kb$-terms was already treated in~\cite[Theorem
5.1]{Costa&Nogueira&Teixeira:2015}. So, we assume without loss of generality that $\rk \alpha=2$ and use those results in rank 1 to manage the new situation, avoiding the difficult technicalities  of the proof (see Section~\ref{subsection:local_groups}
and~\cite{Costa&Nogueira&Teixeira:2015} for more details and missing definitions). We begin by using the same method  to build a finite local group $S_{\alpha,\beta}$ of the form
$S_{\alpha,\beta}=\mathcal{S}(G,L,\mathsf f)$ as follows.

 As $\widetilde{\mathsf w}_\mathbbold{q}(\alpha)\neq\widetilde{\mathsf w}_\mathbbold{q}(\beta)$ by our assumptions, there exists a
finite group $G$  (whose orders of its elements can be chosen arbitrarily large) that fails the identity ${\mathsf w}_\mathbbold{q}(\alpha)={\mathsf w}_\mathbbold{q}(\beta)$.
Hence, there is a group homomorphism $\eta:\FV\rightarrow G$ such  that $\eta({\mathsf
w}_\mathbbold{q}(\alpha))\neq \eta({\mathsf w}_\mathbbold{q}(\beta))$. For each variable
$\mathsf{v}_{*}$ of ${\mathsf V}$ occurring in ${\mathsf w}_\mathbbold{q}(\alpha)$ or ${\mathsf
w}_\mathbbold{q}(\beta)$, denote $\eta(\mathsf{v}_{*})$ by $g_{*}$. Since $\eta$ is a group homomorphism,
 by~\eqref{eq:outline_initial_final_decomposition},
\begin{equation}\label{eq:g_alpha_decomp}
\eta({\mathsf w}_\mathbbold{q}(\alpha))=g_{u,x}\eta(\bw_\mathbbold{q}(\alpha))g_{y,v}\quad\mbox{ and }\quad
\eta({\mathsf w}_\mathbbold{q}(\beta))=g_{u,x}\eta(\bw_\mathbbold{q}(\beta))g_{y,v}.
\end{equation}
Next, let $L$ and ${\mathsf f}$ be the ones that would be chosen by the process of~\cite[Theorem 5.1]{Costa&Nogueira&Teixeira:2015} for the rank 1 canonical forms $\alpha_1$ and $\beta_1$ such that $\alpha_1=\alpha^{(2)}$ and $\beta_1=\beta^{(2)}$
 when $\rk \beta=2$ or $\beta_1=\beta$ when $\rk \beta=1$. This completes the definition of the semigroup $S_{\alpha,\beta}=\mathcal{S}(G,L,{\mathsf f})$.

 Since $S_{\alpha,\beta}$ is a finite semigroup, there is a positive integer $\ell>2$ such that  $s^\omega=s^\ell$ for every $s\in S_{\alpha,\beta}$.
 In particular, as $G$ is isomorphic to a subgroup of $S_{\alpha,\beta}$, $g^\ell=1_G$ for all $g\in G$. Let $\widehat\alpha=\alpha^{(\ell-1)}$
 and let $\widehat\beta= \beta^{(\ell-1)}$  in case $\rk\beta=2$ and
  $\widehat\beta=\beta$ otherwise.  Therefore, since $S_{\alpha,\beta}\in \LG$  and $\LG\models\alpha=\beta$, $S_{\alpha,\beta}$
 satisfies $\widehat\alpha=\alpha=\beta=\widehat\beta$. On the other hand, $\mathbbold{q}_{\widehat\alpha}=\mathbbold{q}_\alpha$ and
 $\mathbbold{q}_{\widehat\beta}=\mathbbold{q}_\beta$, so that
 $\mathbbold{q}\geq {\rm max}\{\mathbbold{q}_{\widehat\alpha},\mathbbold{q}_{\widehat\beta}\}$.
  By the choice of $\ell$ and since $\eta$ is a group homomorphism, one can verify easily from the definition of $\mathbbold{q}$-outline that  the equalities $\eta(\bw_\mathbbold{q}(\widehat\alpha))=\eta(\bw_\mathbbold{q}(\alpha))$ and $\eta(\bw_\mathbbold{q}(\widehat\beta))=\eta(\bw_\mathbbold{q}(\beta))$ hold.

Now, let $\phi: T_A^{\kb} \rightarrow S_{\alpha,\beta}$ be the homomorphism of $\kb$-semigroups
defined by $\phi(a)=a$ for $a\in A$. Since $\alpha_1$ and $\widehat\alpha$ (resp.\ $\beta_1$ and
$\widehat\beta$) have the same  portions and the parameters $L$ and $\mathsf{f}$ of the semigroup
$S_{\alpha,\beta}=\mathcal{S}(G,L,\mathsf f)$ depend only on those portions and on the homomorphism $\eta$, one
can verify by the proof of~\cite[Theorem 5.1]{Costa&Nogueira&Teixeira:2015} that $\phi(\widehat\alpha)$ and
$\phi(\widehat\beta)$ are triples of the form $(\_\; ,h_0\eta(\bw_\mathbbold{q}(\widehat\alpha))h_1,\_)$ and
$(\_\; ,h_0\eta(\bw_\mathbbold{q}(\widehat\beta))h_1,\_)$ where $h_0$ is $g_{x}$  when $u\neq \epsilon$ and it is
$1_G$ otherwise, and $h_1$ is $g_{y}$  when $v\neq \epsilon$ and it is $1_G$ otherwise. Since $S_{\alpha,\beta}$
satisfies $\widehat\alpha=\widehat\beta$, it follows that
$\eta(\bw_\mathbbold{q}(\widehat\alpha))=\eta(\bw_\mathbbold{q}(\widehat\beta))$.
  As $\eta(\bw_\mathbbold{q}(\widehat\alpha))=\eta(\bw_\mathbbold{q}(\alpha))$ and $\eta(\bw_\mathbbold{q}(\widehat\beta))=\eta(\bw_\mathbbold{q}(\beta))$, it follows that $\eta(\bw_\mathbbold{q}(\alpha))=\eta(\bw_\mathbbold{q}(\beta))$, whence, by~\eqref{eq:g_alpha_decomp}, $\eta({\mathsf w}_\mathbbold{q}(\alpha))=\eta({\mathsf w}_\mathbbold{q}(\beta))$. However, we affirmed above that $\eta({\mathsf w}_\mathbbold{q}(\alpha))\neq \eta({\mathsf w}_\mathbbold{q}(\beta))$ as a consequence of assuming that $\widetilde{\mathsf w}_\mathbbold{q}(\alpha)\neq\widetilde{\mathsf w}_\mathbbold{q}(\beta)$. Hence, this condition does not hold, thus concluding the proof of the proposition.
\end{proof}

An immediate consequence of Proposition~\ref{prop:nec_condit_eq_kterms} is that, for any rank $2$ $\kb$-term
$\alpha$ in semi-canonical form,
 $\widetilde{\mathsf w}_\mathbbold{q}(\alpha)=\widetilde{\mathsf w}_\mathbbold{q}(\alpha')$,
 where $\alpha'$ is the canonical form of $\alpha$.

\subsection{Properties of the $\mathbbold{q}$-root of a $\kb$-term}\label{section:Properties of root}

In the remaining of the paper,  when a rank 2 semi-canonical $\kb$-term $\alpha$ is given, we will usually consider its rank configuration of the form
\begin{equation}\label{eq:alpha_rank2_nf}
\alpha=\alpha_0\alpha_1^{\omega-1}\alpha_2\cdots \alpha_{2m-1}^{\omega-1}\alpha_{2m}.
\end{equation}
Notice that the $\mathbbold{q}$-outline ${\mathsf w}_\mathbbold{q}(\alpha)$ may be
written as
$${\mathsf w}_\mathbbold{q}(\alpha)=\mathsf{w}_{\alpha,0}\mathsf{w}_{\alpha,1}\mathsf{w}_{\alpha,2}\cdots \mathsf{w}_{\alpha,2m-1}
\mathsf{w}_{\alpha,2m}$$
 where:  $\mathsf{w}_{\alpha,2i-1} =\bw_\mathbbold{q}(\alpha_{2i-1}^{\omega-1})$ is a word on $\mathsf{V}^{-1}$ for each odd index $2i-1\in\{1,3,\ldots,2m-1\}$;  $\mathsf{w}_{\alpha,2i'}$ is a non-empty word on $\mathsf{V}$ for each even index
$2i'\in\{0,2,\ldots,2m\}$.  We then call each $\mathsf{w}_{\alpha,2i-1}$ a \emph{negative block} and each $\mathsf{w}_{\alpha,2i'}$ a \emph{positive
block} of  ${\mathsf w}_\mathbbold{q}(\alpha)$. Observe that, in each $\mathsf{w}_{\alpha,j}$ ($j\in\{0,1,\ldots, 2m\}$), crucial variables alternate with powers of base
variables. More precisely, for an odd $j$ the alternation is of the form
 $\mathsf{c}^{-1}_{x,\underline{\ },\underline{\ }} \mathsf{b}_{x}^{-r}\mathsf{c}^{-1}_{\underline{\ },\underline{\ },x }$, and for an even $j$ it is of the form $\mathsf{c}_{\underline{\ },\underline{\ },x} \mathsf{b}_{x}^{r}\mathsf{c}_{x,\underline{\ },\underline{\ } }$, where $r$ is a positive integer.
Moreover, $\mathsf{w}_{\alpha,j}$ begins and ends with a crucial variable except for $j=0$, in
which case it begins with the initial variable $\mathsf{i}(\alpha)$, and for $j=2m$, in which case it ends with the final variable $\mathsf{t}(\alpha)$.

Although, for the calculation of the $\mathbbold{q}$-root $\widetilde{\mathsf w}_\mathbbold{q}(\alpha)$,
the occurrences of \emph{spurs} (i.e., products of the form $vv^{-1}$ or  $v^{-1}v$ with $v\in \mathsf V$) in ${\mathsf
w}_\mathbbold{q}(\alpha)$ may be canceled in any order, we will assume that each cancelation step consists in
deleting the leftmost occurrence of a spur. With this assumption, the process of cancelation of ${\mathsf
w}_\alpha$ transforms each block $\mathsf{w}_{\alpha,j}$ into a unique and well-determined (possibly empty)
word,  called the \emph{remainder} of ${\mathsf w}_{\alpha,j}$ and denoted $\mathsf{r}_{\alpha,j}$, so that
 $$\widetilde{\mathsf w}_\mathbbold{q}(\alpha)={\mathsf r}_{\alpha,0}{\mathsf r}_{\alpha,1}{\mathsf r}_{\alpha,2}\cdots {\mathsf r}_{\alpha,2m-1} {\mathsf
r}_{\alpha,2m}.$$ In particular, the reduction process can, possibly,  eliminate completely some of the negative
blocks of ${\mathsf w}_\mathbbold{q}(\alpha)$ or gather into a unique negative block of $\widetilde{\mathsf w}_\mathbbold{q}(\alpha)$ some factors occurring in distinct negative blocks of ${\mathsf w}_\mathbbold{q}(\alpha)$, in which case the intermediate positive blocks are completely deleted.

 For a finite word ${\mathsf w}$ over the alphabet $\mathsf{V}\cup \mathsf{V}^{-1}$, we define the \emph{crucial length} of ${\mathsf w}$ as  the number of  occurrences of crucial variables in ${\mathsf w}$, and  denote it by $\crl{\mathsf w}$.  For each $j\in\{0,1,\ldots ,2m\}$, we
denote by $\dd_{\alpha,j}$  the number of occurrences of crucial variables in ${\mathsf w}_{\alpha,j}$ that are
canceled in the computation of $\widetilde{\mathsf w}_\mathbbold{q}(\alpha)$, that is, $$\dd_{\alpha,j}=\crl
{\mathsf{w}_{\alpha,j}}-\crl {\mathsf{r}_{\alpha,j}}.$$
Note that $\crl{\mathsf w_{\alpha,j}}$ is  the $1$-length of $\alpha_j$ in case $j\in\{0,2m\}$ and it is equal to the 1-length  of $\alpha_j$ plus one otherwise. Since the cancelations in ${\mathsf w}_{\alpha,j}$ are
performed from the extremes, ${\mathsf w}_{\alpha,j}=\wlslash_{\alpha,j}{\mathsf
r}_{\alpha,j}\wrslash_{\alpha,j}$ where $\wlslash_{\alpha,j}$ (resp.\ $\wrslash_{\alpha,j}$) is the longest
prefix (resp.\ suffix) of ${\mathsf w}_{\alpha,j}$ that is canceled by variables occurring on its left side
(resp.\ right side). The following lateral versions of $\dd_{\alpha,j}$ will be convenient. We let
$$\dl_{\alpha,j}=\crl{\wlslash_{\alpha,j}},\quad \dr_{\alpha,j}=\crl {\wrslash_{\alpha,j}},$$
and notice that $\dd_{\alpha,j}=\dl_{\alpha,j}+\mathbbold{c}\hspace*{-1pt}\mbox{\textquoteleft\!}_{\alpha,j}$ and $\dl_{\alpha,j}=0$ (resp.\ $\dr_{\alpha,j}=0$) if and only if $\wlslash_{\alpha,j}=\epsilon$ (resp.\ $\wrslash_{\alpha,j}=\epsilon$) since each intermediate block begins and ends with a crucial variable.

The following lemma presents important properties of the $\mathbbold{q}$-root of a rank 2 $\kb$-term which is a
 canonical form.
\begin{lemma}\label{lemma:oc_crucial_variables}
Let $\alpha$ be a rank $2$ canonical form with rank configuration of the form~\eqref{eq:alpha_rank2_nf} and let
$j\in\{1,2,\ldots,2m-1\}$.
\begin{enumerate}
\item\label{item:oc_crucial_variables-a} If $j$ is odd, then $\dl_{\alpha,j}\le 2$ and $\mathbbold{c}\hspace*{-2.7pt}\mbox{\textquoteleft\hspace*{-0.1pt}}_{\alpha,j}\le 1$ with $\dd_{\alpha,j}\le 2$.
\item\label{item:oc_crucial_variables-b} $\crl{{\mathsf r}_{\alpha,j}}\neq 0$.
\end{enumerate}
 Note that, by~\ref{item:oc_crucial_variables-b},  ${\mathsf r}_{\alpha,j}$ is non-empty for all
$j\in\{1,2,\ldots,2m-1\}$. Therefore, the number  of negative blocks of $\widetilde{\mathsf w}_\mathbbold{q}(\alpha)$ is equal to
the $2$-length $m$ of $\alpha$. Moreover,  the cancelation of the prefix $\wlslash_{\alpha,j}$ (resp.\ the
suffix $\wrslash_{\alpha,j}$) of ${\mathsf w}_{\alpha,j}$ is caused only by the adjacent block ${\mathsf
w}_{\alpha,j-1}$ (resp.\ ${\mathsf w}_{\alpha,j+1}$). That is, informally speaking, each block has only a ``local influence''. This means that, for each $j\in\{1,2,\ldots,2m\}$, $\wrslash_{\alpha,j-1}$ and $\wlslash_{\alpha,j}$ are mutually inverse words in $\FV$ and, therefore, $\dr_{\alpha,j-1}=\dl_{\alpha,j}$.
\end{lemma}
\begin{proof}
  The proof is made by induction on $m$. Assume first that $m=1$ and so $j=1$, $\alpha=\alpha_0\alpha_1^{\omega-1}\alpha_2$ and
  ${\mathsf w}_\mathbbold{q}(\alpha)={\mathsf w}_{\alpha,0}{\mathsf w}_{\alpha,1}{\mathsf w}_{\alpha,2}$.
  Let $\alpha_1=u_{0}x_{1}^{\omega+q_{1}} u_{1}\cdots x_{n}^{\omega+q_{n}}u_{n}$ be the rank configuration
  of $\alpha_1$, whence
  $${\mathsf w}_{\alpha,1}=\mathsf{c}_{x_{n}, u_{n}u_{0},x_{1}}^{-1} \mathsf{b}_{x_{n}}^{-\mathbbold q_{n}}\mathsf{c}_{x_{n-1}, u_{n-1},x_{n}}^{-1}
   \cdots \mathsf{b}_{x_{2}}^{-\mathbbold q_{2}}\mathsf{c}_{x_{1}, u_{1},x_{2}}^{-1} \mathsf{b}_{x_{1}}^{-\mathbbold q_{1}}
   \mathsf{c}_{x_{n}, u_{n}u_{0},x_{1}}^{-1}.$$
Supposing that $\alpha_1$ is a generic rank 1 $\kb$-term with $n>1$ and $q_n=0$, we define the term
$x_{n-1}^\omega u_{n-1}x_{n}^{\omega}u_{n}$ to be the \emph{final $\omega 2$-portion of $\alpha_1$}. To prove
condition~\ref{item:oc_crucial_variables-a}, we consider two cases.

\begin{enumerate}[label=\textbf{Case \arabic*.},ref=\arabic*,labelwidth=-10ex,labelsep=*,leftmargin=3ex,itemindent=7ex,parsep=1ex]
\item $\alpha_2$ is a rank 1 $\kb$-term with initial $\omega$-portion $u_{0}x_{1}^{\omega}$. In this case, by Steps (2) and (3) of the canonical form reduction algorithm,   $u_0=\epsilon$, $q_1=0$ and
 $\alpha_2$ is of the form $\alpha_2=x_{1}^{\omega+p}\alpha'_2$ with $p\neq 0$ (since in case $p=0$ it would be possible to apply a rank 2 shift right on $\alpha$).  On the other hand,
  ${\mathsf w}_{\alpha,2}=\mathsf{c}(\alpha_1,\alpha_2)\bw_\mathbbold{q}(\alpha_2)\mathsf{t}(\alpha_2)$, whence ${\mathsf w}_{\alpha,2}$ is of the form ${\mathsf w}_{\alpha,2}=\mathsf{c}_{x_{n}, u_{n},x_{1}}\mathsf{b}_{x_{1}}^{\mathbbold p}{\mathsf w}'_{\alpha,2}$. Therefore $\wrslash_{\alpha,1}=\mathsf{b}_{x_{1}}^{-p'}\mathsf{c}_{x_{n}, u_{n},x_{1}}^{-1}$ (and $\wlslash_{\alpha,2}=\mathsf{c}_{x_{n}, u_{n},x_{1}}\mathsf{b}_{x_{1}}^{p'}$) where $p'$ is $\mathbbold q$ when $p>0$ and it is $\mathbbold q+p$ when $p<0$. It follows that
$\dr_{\alpha,1}= 1$. The following subcases may happen.

\begin{enumerate}[label=\textbf{Case 1.\arabic*.},ref=\arabic*,labelwidth=-10ex,labelsep=*,leftmargin=2ex,itemindent=12ex,parsep=1ex]
    \item $\alpha_0$ is a rank 1 $\kb$-term with final $\omega$-portion $x_{n}^{\omega}u_{n}$. By Step (3), one deduces
that $q_n=0$,  $\alpha_0=\alpha'_0 x_{n}^{\omega+r}u_{n}$ with $r\in\mathbb Z$ and $u_n=u_{x_n,x_1}$. Notice that, in this case, $n>1$. Indeed, if $n$ was 1, then the limit term $\alpha_1^{\omega-1}$ would
be eliminated in Step (1) of the canonical form reduction algorithm. If $r\neq 0$, then one derives
$\dl_{\alpha,1}= 1$ as above and concludes that $\dd_{\alpha,1}= 2$. Suppose now that $r=0$ and notice that
$x_{n-1}^{\omega}u_{n-1}x_{n}^{\omega}u_{n}$ can not be the final $\omega 2$-portion of $\alpha_0$. Indeed,
otherwise, by Step (3) of the canonical form reduction algorithm, it would be possible to shorten the rank 2 limit term.
     As a consequence,
 $\mathsf{c}_{x_{n-1}, u_{n-1},x_{n}} \mathsf{b}_{x_{n}}^{\mathbbold q}\mathsf{c}_{x_{n}, u_{n},x_{1}}$ is not a
 suffix of ${\mathsf w}_{\alpha,0}$ and, so, the equalities $\dl_{\alpha,1}= 1$ and $\dd_{\alpha,1}= 2$ also hold for $r=0$.

\item  $x_{n}^{\omega}u_{n}$ is not the final $\omega$-portion of $\alpha_0$. In this case, it is immediate that
$\mathsf{c}_{x_{n}, u_{n},x_{1}}$ is not the final variable of ${\mathsf w}_{\alpha,0}$. Therefore,
$\dl_{\alpha,1}= 0$ and $\dd_{\alpha,1}= 1$.
\end{enumerate}
\item $\alpha_2$ has not $u_{0}x_{1}^{\omega}$ as initial $\omega$-portion. Then,
we deduce readily that $\dr_{\alpha,1}= 0$ and,  as in Case 1., consider two subcases.
\begin{enumerate}[label=\textbf{Case 2.\arabic*.},ref=\arabic*,labelwidth=-10ex,labelsep=*,leftmargin=2ex,itemindent=12ex,parsep=1ex]
    \item  $\alpha_0$ is a rank 1 $\kb$-term with final $\omega$-portion $x_{n}^{\omega}u_{n}$. In this case, Step (3)
 determines also  $q_n=0$ and $\alpha_0=\alpha'_0x_{n}^{\omega+r}u_{n}$. If $r\neq 0$, then
$\dl_{\alpha,1}= 1$ and so $\dd_{\alpha,1}= 1$. Let now $r=0$. If $n=1$, then $\alpha=\alpha'_0x_1^\omega u_1 (u_0
x_{1}^{\omega}u_{1})^{\omega-1}\alpha_2$ and, as above, $x_{1}^{\omega}u_{1}u_{0}$  cannot be  the final $\omega$-portion of $\alpha'_0$  since otherwise $\alpha$ could be reduced to a rank 1 $\kb$-term in Step (1). So, $\dd_{\alpha,1}=\dl_{\alpha,1}=1$ in that case. Assume now $n>1$.
If $x_{n-1}^{\omega}u_{n-1}x_n^\omega u_n$
 is the final $\omega 2$-portion of
 $\alpha_0$, then $q_{n-1}=0$ and $\dl_{\alpha,1}\geq 2$.
 On the other hand, $x_{n-2}^{\omega}u_{n-2}x_{n-1}^{\omega}u_{n-1}$ ($x_{2}^{\omega}u_{2}x_{1}^{\omega}u_{1}$ in case $n=2$) cannot be the final $\omega2$-portion of $\alpha'_0$ since, otherwise, as in Case 1.1., it  would be  possible to apply a type 2 shortening. Whence $\dd_{\alpha,1}=\dl_{\alpha,1}=2$.

\item $x_{n}^{\omega}u_{n}$ is not the final $\omega$-portion of $\alpha_0$. In this case, $\dd_{\alpha,1}=\dl_{\alpha,1}= 0$.
\end{enumerate}
\end{enumerate}
The above analysis shows that, in all possible cases, $\dl_{\alpha,j}\le 2$ and $\dr_{\alpha,j}\le 1$ with
$\dd_{\alpha,j}\le 2$, thus proving~\ref{item:oc_crucial_variables-a} for $m=1$.

Condition~\ref{item:oc_crucial_variables-b} follows easily from~\ref{item:oc_crucial_variables-a}. Indeed,
  by~\ref{item:oc_crucial_variables-a},
$\crl{{\mathsf r}_{\alpha,1}}= 0$ if and only if $\crl{{\mathsf w}_{\alpha,1}}=\dd_{\alpha,1}=2$, in which case
$n=1$. This excludes Case 1.1.\ and Case 2.1.\ with $n>1$, the only
situations in which $\dd_{\alpha,1}=2$, and allows us to conclude that $\crl{{\mathsf r}_{\alpha,1}}>0$, thus
proving~\ref{item:oc_crucial_variables-b} for $m=1$.

Let now $m>1$ and suppose, by induction hypothesis, that the result holds for $\kb$-terms with 2-length at most
$m-1$.
 Let $\vec\alpha=\alpha_{0}\alpha_{1}^{\omega-1}\alpha_{2}\cdots \alpha_{2m-3}^{\omega-1}\alpha_{2m-2} ux^\omega$ and $\cev{\alpha}=y^\omega v\alpha_{2m-2} \alpha_{2m-1}^{\omega-1}\alpha_{2m}$, where $ux^\omega$ and $y^\omega v$ are, respectively, the initial $\omega$-portion of $\alpha_{2m-1}$ and the final $\omega$-portion of $\alpha_{2m-3}$. Hence
\begin{alignat*}{2}
{\mathsf w}_\mathbbold{q}({\alpha})&={\mathsf w}_{\alpha,0}{\mathsf w}_{\alpha,1}{\mathsf w}_{\alpha,2}{\mathsf
w}_{\alpha,3}\cdots {\mathsf w}_{\alpha,2m}\\[1mm]
 {\mathsf w}_\mathbbold{q}(\vec\alpha)&={\mathsf w}_{\vec\alpha,0}{\mathsf w}_{\vec\alpha,1} \cdots {\mathsf w}_{\vec\alpha,2m-2}  =
 {\mathsf w}_{\alpha,0}{\mathsf w}_{\alpha,1}\cdots {\mathsf w}_{\alpha,2m-2}\mathsf{b}_x^\mathbbold{q}\mathsf{t}_{x,\epsilon}\\[1mm]
{\mathsf w}_\mathbbold{q}(\cev\alpha)&={\mathsf w}_{\scriptsize\cev\alpha,0}{\mathsf
w}_{\scriptsize\cev\alpha,1}   {\mathsf w}_{\scriptsize\cev\alpha,2}=\mathsf{i}_{\epsilon,y}\mathsf{b}_y^\mathbbold{q}{\mathsf w}_{\alpha,2m-2}{\mathsf
w}_{\alpha,2m-1} {\mathsf w}_{\alpha,2m}.
\end{alignat*}
The $\kb$-term $\vec\alpha$ is clearly a canonical form, while $\cev\alpha$ may not be.   However, this only can happen  when  $v=v'v''$, with $v''\neq \epsilon$ and  $\sigma= v''\alpha_{2m-2}$  a rank 1 $\kb$-term for which there is a $\kb$-term $\tau$ such that  $\alpha_{2m-1}=\tau\sigma$ and $\LI\models \tau=\sigma$. Note that in such case  $\dr_{\alpha,2m-1}=0$ since $v''\in A^+$  is a prefix of $\alpha_{2m-1}$.  Furthermore,  $\cev\alpha$ is of the form $y^\omega v' \sigma (\tau\sigma)^{\omega-1}\alpha_{2m}$ and  its  canonical form is $\cev\alpha'= y^\omega v' \tau^{\omega-1}\alpha_{2m}$.   The respective $\mathbbold{q}$-outline ${\mathsf w}_\mathbbold{q}(\cev\alpha')$  is  such that
$${\mathsf w}_\mathbbold{q}(\cev\alpha')={\mathsf w}_{\scriptsize\cev\alpha',0}  {\mathsf w}_{\scriptsize\cev\alpha',1} {\mathsf w}_{\scriptsize\cev\alpha',2}= {\mathsf r}_{\scriptsize\cev\alpha',0}  {\mathsf r}_{\scriptsize\cev\alpha',1}  \wrslash_{\scriptsize\cev\alpha',1}  \wlslash_{\scriptsize\cev\alpha',2} {\mathsf r}_{\scriptsize\cev\alpha',2}, $$
and $\crl{{\mathsf r}_{\scriptsize\cev\alpha',0}}=\crl{{\mathsf w}_{\scriptsize\cev\alpha',0}}=1$.   By Proposition~\ref{prop:nec_condit_eq_kterms},  $\widetilde{{\mathsf w}}_\mathbbold{q}(\cev\alpha') =  \widetilde{{\mathsf w}}_\mathbbold{q}(\cev\alpha)$, and so
  ${\mathsf r}_{\scriptsize\cev\alpha',i}={\mathsf r}_{\scriptsize\cev\alpha,i}$ for $i=0,1,2$.

By the induction hypothesis, the statement  holds for both $\vec\alpha$  and $\cev\alpha'$, where $\cev\alpha'$
is taken to be $\cev\alpha$ in case $\cev\alpha$ is a canonical form. In particular, the occurrences of crucial
variables in  ${\mathsf w}_{\vec\alpha,2m-3}$ ($={\mathsf w}_{\alpha,2m-3}$  ) are not all canceled in the
simplification of ${\mathsf w}_\mathbbold{q}(\vec\alpha)$ in the free group $\FV$, and so  $\crl{ {\mathsf
r}_{\scriptsize\vec\alpha,2m-3}  }\ge 1$ . Analogously, there exist occurrences of crucial variables in
${\mathsf w}_{\scriptsize\cev\alpha',1}$    that are not canceled in the reduction of  ${\mathsf
w}_\mathbbold{q}(\cev\alpha')$, which implies that  $\crl{ {\mathsf r}_{\scriptsize\cev\alpha,1}  }\ge 1$ since
$\crl{ {\mathsf r}_{\scriptsize\cev\alpha,1}  }=\crl{{\mathsf r}_{\scriptsize\cev\alpha',1}  }$.
  Putting together these two facts, we deduce that  $\crl{{\mathsf r}_{\alpha,2m-3}}$ and $\crl{{\mathsf r}_{\alpha,2m-1}}$ are both positive, thus showing, in particular, that each  block has only a ``local influence'' in the reduction process.  Furthermore,  ${\mathsf r}_{\vec\alpha,2m-3}={\mathsf r}_{\alpha,2m-3}$, because we begin deleting the leftmost spurs, and $\dl_{\alpha,2m-1}\le \dl_{\scriptsize\cev\alpha,1}$.   Therefore, statement~\ref{item:oc_crucial_variables-a} follows immediately from the induction hypothesis applied to $\vec\alpha$  and $\cev\alpha$. To conclude the proof of statement~\ref{item:oc_crucial_variables-b}, and of the lemma, it remains to show that $\crl{{\mathsf r}_{\alpha,2m-2}}\neq 0$,    and so that $\dl_{\alpha,2m-1}=\dl_{\scriptsize\cev\alpha,1}$. We know already that the cancelations on ${\mathsf w}_{\alpha,2m-2}$ are determined only by the adjacent blocks ${\mathsf
w}_{\alpha,2m-3}$ and ${\mathsf w}_{\alpha,2m-1}$. So, it suffices to consider the subterm
$\alpha_{2m-3,2m-1}=\alpha_{2m-3}^{\omega-1}\alpha_{2m-2}\alpha_{2m-1}^{\omega-1}$ of $\alpha$ which, as one
recalls, is a canonical form.

To begin with, notice that $\crl{{\mathsf w}_{\alpha,2m-2}}=\ell+1$ where $\ell$ is the 1-length of $\alpha_{2m-2}$. On
the other hand, by~\ref{item:oc_crucial_variables-a}, $\dl_{\alpha,{2m-2}}=\dr_{\alpha,{2m-3}}\le 1$ and
$\dr_{\alpha,{2m-2}}=\dl_{\alpha,{2m-1}}\le 2$ so that $\dd_{\alpha,{2m-2}}\le 3$. Suppose by way of contradiction that
$\crl{\mathsf{r}_{\alpha,2m-2}}=0$ and, so, that $\ell\leq 2$. Let us analyse, for each of the three possible
values of $\ell$, what could hypothetically be the forms of $\alpha_{2m-3,2m-1}$ and verify that, actually, those
possibilities are not compatible with $\alpha_{2m-3,2m-1}$ being a canonical form.
\begin{enumerate}[label=\arabic*)]
\item $\ell=0$, that is, $\alpha_{2m-2}=w_0\in A^*$.  In this case $\crl{{\mathsf w}_{\alpha,{2m-2}}}=1$ and so, by hypothesis, $\dd_{\alpha,2m-2}=1$. Hence,
either $\dl_{\alpha,2m-2}=1$ and $\dr_{\alpha,2m-2}=0$, or $\dl_{\alpha,2m-2}=0$ and $\dr_{\alpha,2m-2}=1$. Then
$\alpha_{2m-3,2m-1}$ is of one of the forms $\alpha_{2m-3,2m-1}=( w_0ux^{\omega+p}\rho_1)^{\omega-1}w_0
(ux^{\omega+q}\rho_3)^{\omega-1}$ or $\alpha_{2m-3,2m-1}=(\rho_1 y^{\omega+p}v)^{\omega-1}w_0
(\rho_3 y^{\omega+q}vw_0)^{\omega-1}$.

\item $\ell=1$, say with $\alpha_{2m-2}=w_0z_1^{\omega+q_1}w_1$.  Then $\crl{{\mathsf w}_{\alpha,2m-2}}=\dd_{\alpha,2m-2}=2$ and either
$\dl_{\alpha,2m-2}=1$ and $\dr_{\alpha,2m-2}=1$, or $\dl_{\alpha,2m-2}=0$ and $\dr_{\alpha,2m-2}=2$. In this circumstance,
$\alpha_{2m-3,2m-1}$ is of one of the forms $\alpha_{2m-3,2m-1}=(z_1^{\omega}\rho_1)^{\omega-1}z_1^{\omega+q_1}w_1
(\rho_3 z_1^{\omega}w_1)^{\omega-1}$,  in which case $w_0$ must be empty, or $\alpha_{2m-3,2m-1}=(\rho_1y^{\omega+p}v)^{\omega-1}w_0z_1^{\omega}w_1
(\rho_3y^{\omega+r}vw_0z_1^{\omega}w_1)^{\omega-1}$, in which case $q_1=0$.

\item $\ell=2$, with $\alpha_{2m-2}=w_0z_1^{\omega+q_1}w_1z_2^{\omega+q_2}w_3$.  Hence $\crl{{\mathsf w}_{\alpha,2m-2}}=\dd_{\alpha,2m-2}=3$
with $\dl_{\alpha,2m-2}=1$ and $\dr_{\alpha,2m-2}=2$. In this case   $w_0=\epsilon$, $q_2=0$  and  $\alpha_{2m-3,2m-1}$ is of the form $\alpha_{2m-3,2m-1}=(z_1^{\omega}\rho_1)^{\omega-1}z_1^{\omega+q_1}w_1z_2^{\omega}w_2
(\rho_3z_1^{\omega}w_1z_2^{\omega}w_2)^{\omega-1}$.
\end{enumerate}
In all of the above situations it is possible to make a limit term agglutination on $\alpha_{2m-3,2m-1}$ and, so, this
$\kb$-term is not a canonical form by Step~(1) of the rank 2 reduction algorithm. Consequently,
$\crl{\mathsf{r}_{\alpha,2m-2}}>0$ and the proof is complete.
\end{proof}

 It is useful, for later reference, to state the following facts shown in the proof of Lemma~\ref{lemma:oc_crucial_variables}.
\begin{remark}\label{remark:cancelations}
 For an integer $p$ let $p'$ denote $\mathbbold{q}$ when $p\geq 0$ and let it denote $\mathbbold{q}+p$ otherwise. Let $\alpha$ be a canonical rank 2 $\kb$-term of the form~\eqref{eq:alpha_rank2_nf}, let $j$ be an odd position and let $\alpha_j= u_{0}x_{1}^{\omega+q_{1}} u_{1}\cdots x_{n}^{\omega+q_{n}}u_{n}$.  Then,
\begin{enumerate}
  \item $\dr_{\alpha,j}=1$ if and only if $u_0=\epsilon$, $q_1=0$ and $\alpha_{j+1}$ is of the form $\alpha_{j+1}= x_1^{\omega+p}\alpha'_{j+1}$ with $p\neq 0$. Moreover, in this case, $\wrslash_{\alpha,j}= \mathsf{b}_{x_{1}}^{-p'}  {\mathsf c}^{-1}_{x_{n},u_{n},x_{1}}$.
\item\label{previous} $\dl_{\alpha,j}=2$ if and only if  $n>1$, $q_{n-1}=q_n=0$ and $\alpha_{j-1}$ is of the form $\alpha_{j-1}= \alpha'_{j-1}x_{n-1}^{\omega+p} u_{x_{n-1},x_n} x_n^\omega u_n$. In this case, $\wlslash_{\alpha,j}=  {\mathsf c}^{-1}_{x_{n},u_{n}u_0,x_{1}}\mathsf{b}_{x_{n}}^{-\mathbbold{q}} {\mathsf c}^{-1}_{x_{n-1},u_{x_{n-1},x_{n}},x_{n}}\mathsf{b}_{x_{n-1}}^{-p'}$.
\item $\dl_{\alpha,j}=1$  if and only if  $q_n=0$,  $\alpha_{j-1}= \alpha'_{j-1}x_{n}^{\omega+p} u_n$ and,  when $n>1$,   $x_{n-1}^\omega u_{x_{n-1},x_n} x_n^\omega u_n$ is not the final $\omega 2$-portion of $\alpha_{j-1}$. In this case, $\wlslash_{\alpha,j}=  {\mathsf c}^{-1}_{x_{n},u_{n}u_0,x_{1}}\mathsf{b}_{x_{n}}^{-p'}$.
  \item for $\dr_{\alpha,j}=\dl_{\alpha,j}=1$,  $u_n=u_{x_n,x_1}$.
\end{enumerate}
\end{remark}

We can deduce already  a weaker version of Theorem~\ref{theo:equal_canonical_forms}.
\begin{corollary}\label{cor:2-length_of_cf}
Let $\alpha$ and $\beta$ be canonical forms such that $\LG\models \alpha=\beta$.
\begin{enumerate}
\item\label{item:2-length_of_cf-a} The $\kb$-terms $\alpha$ and $\beta$ have the same rank.
\item\label{item:2-length_of_cf-b} If $\rk\alpha\le 1$, then $\alpha=\beta$.
\item\label{item:2-length_of_cf-c} If $\rk\alpha=2$, then $\alpha$ and $\beta$ have the same $2$-length.
\end{enumerate}
\end{corollary}
\begin{proof} By hypothesis $\LG\models \alpha=\beta$. Hence, as \LI\ is a subpseudovariety of \LG\ that separates
different finite words and finite words from infinite pseudowords, if one of $\alpha$ and $\beta$ is a rank 0 $\kb$-term then they
 are the same $\kb$-term. We may therefore assume that $\alpha$ and $\beta$ have at least rank 1. Now, by
 Proposition~\ref{prop:nec_condit_eq_kterms}, $\widetilde{\mathsf w}_\mathbbold{q}(\alpha)=\widetilde{\mathsf w}_\mathbbold{q}(\beta)$. Thus,
 since the $\mathbbold{q}$-root of a rank 1 $\kb$-term is a word from $\mathsf{V}^+$ and, by Lemma~\ref{lemma:oc_crucial_variables}, the
 $\mathbbold{q}$-root of a rank 2 canonical form contains negative blocks, $\alpha$ and $\beta$ must have the same rank. This proves~\ref{item:2-length_of_cf-a}.

 Statement~\ref{item:2-length_of_cf-b} is an immediate consequence of~\ref{item:2-length_of_cf-a} and~\cite[Theorem 5.1]{Costa&Nogueira&Teixeira:2015}, while~\ref{item:2-length_of_cf-c} is a direct application of~\ref{item:2-length_of_cf-a} and Lemma~\ref{lemma:oc_crucial_variables}.
\end{proof}

As a consequence of the above result, to complete the proof of Theorem~\ref{theo:equal_canonical_forms} it remains to treat the instance in which $\alpha$ and $\beta$ are both rank 2 and have the same $2$-length.  This will be done in the next result.
\begin{proposition}\label{prop:suff_condit_eq_kterms}
Let $\alpha$ and $\beta$ be rank $2$ canonical forms with the same $2$-length. If
 $\widetilde{\mathsf w}_\mathbbold{q}(\alpha)=\widetilde{\mathsf w}_\mathbbold{q}(\beta)$, then $\alpha=\beta$.
\end{proposition}
\begin{proof}
Let $\alpha=\alpha_0\alpha_1^{\omega-1}\alpha_2\cdots \alpha_{2m-1}^{\omega-1}\alpha_{2m}$ and $\beta=\beta_0\beta_1^{\omega-1}\beta_2\cdots \beta_{2m-1}^{\omega-1}\beta_{2m}$  be the rank configurations of $\alpha$ and $\beta$ and assume that $\widetilde{\mathsf w}_\mathbbold{q}(\alpha)=\widetilde{\mathsf w}_\mathbbold{q}(\beta)$. Note that all remainders ${\mathsf r}_{\alpha,i}$ and ${\mathsf r}_{\beta,i}$ are non-empty by Lemma~\ref{lemma:oc_crucial_variables}. Therefore, the assumption  $\widetilde{\mathsf w}_\mathbbold{q}(\alpha)=\widetilde{\mathsf w}_\mathbbold{q}(\beta)$ implies that ${\mathsf r}_{\alpha,i}={\mathsf r}_{\beta,i}$ for every $i\in\{0,1,\ldots,2m\}$. Since $\alpha$ and $\beta$ are canonical forms, we observed already that $\alpha=\beta$ if and only if ${\mathsf w}_\mathbbold{q}(\alpha)={\mathsf w}_\mathbbold{q}(\beta)$. On the other hand, ${\mathsf w}_\mathbbold{q}(\alpha)={\mathsf w}_\mathbbold{q}(\beta)$ if and only if ${\mathsf w}_{\alpha,i}={\mathsf w}_{\beta,i}$ for all $i$. Now, recall that, for $\gamma\in\{\alpha,\beta\}$, ${\mathsf w}_{\gamma,i}=\wlslash_{\gamma,i}{\mathsf r}_{\gamma,i}\wrslash_{\gamma,i}$ that, for $i\neq 0$, $\wrslash_{\gamma,i-1}$ and $\wlslash_{\gamma,i}$ are mutually inverse words in $\FV$ and that $\wlslash_{\gamma,0}=\wrslash_{\gamma,2m}=\epsilon$. Therefore, to deduce the equality $\alpha=\beta$ it suffices to prove that, for each odd position $j\in\{1,3,\ldots,2m-1\}$,
\begin{equation}\label{eq:same_cancelations}
 \wlslash_{\alpha,j}=\wlslash_{\beta,j}\mbox{ and }\wrslash_{\alpha,j}=\wrslash_{\beta,j}.
\end{equation}
Throughout, let  $j\in\{1,3,\ldots,2m-1\}$ be an odd integer and let  $\alpha_j=u_{0}x_{1}^{\omega+q_{1}} u_{1}\cdots x_{n}^{\omega+q_{n}}u_{n}$ and $\beta_j=v_{0}y_{1}^{\omega+p_{1}} v_{1}\cdots y_{k}^{\omega+p_{k}}v_{k}$ be the rank configurations of $\alpha_j$ and $\beta_j$. To prove~\eqref{eq:same_cancelations}, let us show first that ${\mathsf w}_{\alpha,j}$ and ${\mathsf w}_{\beta,j}$ admit the same number of right cancelations of occurrences of crucial variables.
\begin{claim}
$\dr_{\alpha,j}=\dr_{\beta,j}$.
\end{claim}
\begin{proof}
 We know from Lemma~\ref{lemma:oc_crucial_variables} that $\dr_{\alpha,j}, \dr_{\beta,j}\in\{0,1\}$. Suppose that  $\dr_{\alpha,j}=1$ and $\dr_{\beta,j}=0$. As observed in Remark~\ref{remark:cancelations} $(a)$, the first equality gives  $u_0=\epsilon$, $q_1=0$ and $\alpha_{j+1}= x_1^{\omega+p}\alpha'_{j+1}$ for some integer $p\neq 0$. Hence  ${\mathsf r}_{\alpha,j}={\mathsf r}'_{\alpha,j}{\mathsf b}_{x_1}^p$ when $p<0$, and  ${\mathsf r}_{\alpha,j+1}={\mathsf b}_{x_1}^p{\mathsf r}'_{\alpha,j+1}$ when $p>0$. The second equality implies that ${\mathsf r}_{\beta,j}$ ends with a crucial variable and that ${\mathsf r}_{\beta,j+1}$ either begins with a crucial variable, or  is equal to the final variable ${\mathsf t}(\beta)$ for  $j+1=2m$ and $\alpha_{2m}\in A^*$. This  contradicts the fact that ${\mathsf r}_{\alpha,i}={\mathsf r}_{\beta,i}$ for all $i$. Therefore $\dr_{\alpha,j}=1$ and $\dr_{\beta,j}=0$ does not apply, and neither does  $\dr_{\alpha,j}=0$ and $\dr_{\beta,j}=1$ by symmetry, thus proving that $\dr_{\alpha,j}= \dr_{\beta,j}$.
\end{proof}

Let us now show  the following:
\begin{claim}
If $\dl_{\alpha,j}=\dl_{\beta,j}$, then $\wlslash_{\alpha,j}=\wlslash_{\beta,j}$ and $\wrslash_{\alpha,j}= \wrslash_{\beta,j}$ (and, so, $\alpha_j=\beta_j$).
\end{claim}
\begin{proof}
Suppose that $\dl_{\alpha,j}=\dl_{\beta,j}$, whence $\dd_{\alpha,j}=\dd_{\beta,j}$ since, by Claim 1, $\dr_{\alpha,j}=\dr_{\beta,j}$. Then, from ${\mathsf r}_{\alpha,j} ={\mathsf r}_{\beta,j}$ it follows that $n=k$ and that ${\mathsf w}_{\alpha,j}$ and ${\mathsf w}_{\beta,j}$ are of the form
\begin{alignat*}{2}
{\mathsf w}_{\alpha,j}&=\mathsf{c}_{x_{n}, u_{n}u_{0},x_{1}}^{-1} \mathsf{b}_{x_{n}}^{-\mathbbold q_{n}}\mathsf{c}_{x_{n-1}, u_{n-1},x_{n}}^{-1}
   \cdots \mathsf{b}_{x_{2}}^{-\mathbbold q_{2}}\mathsf{c}_{x_{1}, u_{1},x_{2}}^{-1} \mathsf{b}_{x_{1}}^{-\mathbbold q_{1}}
   \mathsf{c}_{x_{n}, u_{n}u_{0},x_{1}}^{-1}
   \\[1mm]
{\mathsf w}_{\beta,j}&= \mathsf{c}_{y_{n}, v_{n}v_{0},y_{1}}^{-1} \mathsf{b}_{y_{n}}^{-\mathbbold p_{n}}\mathsf{c}_{y_{n-1}, v_{n-1},y_{n}}^{-1}
   \cdots \mathsf{b}_{y_{2}}^{-\mathbbold p_{2}}\mathsf{c}_{y_{1}, v_{1},y_{2}}^{-1} \mathsf{b}_{y_{1}}^{-\mathbbold p_{1}}
   \mathsf{c}_{y_{n}, v_{n}v_{0},y_{1}}^{-1}.
\end{alignat*}

We begin by showing the equality  $\wrslash_{\alpha,j}= \wrslash_{\beta,j}$ which is easier to prove. If $\dr_{\alpha,j}=0$ then $\wrslash_{\alpha,j}=\epsilon=\wrslash_{\beta,j}$. It remains to consider  $\dr_{\alpha,j}=1$. In this case $\dl_{\alpha,j}\leq 1$ by Lemma~\ref{lemma:oc_crucial_variables} and, by Remark~\ref{remark:cancelations} $(a)$,  $u_0=v_0=\epsilon$, $q_1=p_1=0$, $\alpha_{j+1}= x_1^{\omega+r}\alpha'_{j+1}$, $\beta_{j+1}= x_1^{\omega+s}\beta'_{j+1}$ for some non-zero integers $r$ and $s$,  $\wrslash_{\alpha,j}= \mathsf{b}_{x_{1}}^{-r'}  {\mathsf c}^{-1}_{x_{n},u_{n},x_{1}}$ and  $\wrslash_{\beta,j}= \mathsf{b}_{y_{1}}^{-s'}  {\mathsf c}^{-1}_{y_{n},v_{n},y_{1}}$ where, for $t\in\{r,s\}$, $t'=\mathbbold{q}$ when $t> 0$ and  $t'=\mathbbold{q}+t$ when $t< 0$. So, as ${\mathsf r}_{\alpha,j} ={\mathsf r}_{\beta,j}$, one deduces immediately that $r=s$,  $x_1=y_1$ and $x_n=y_n$. To complete the proof of $\wrslash_{\alpha,j}= \wrslash_{\beta,j}$ it remains to show that $u_n=v_n$. For $\dl_{\alpha,j}=0$, this follows trivially from the equalities ${\mathsf r}_{\alpha,j} ={\mathsf r}_{\beta,j}$ and $u_0=v_0$. In case $\dl_{\alpha,j}=1$, one deduces from Remark~\ref{remark:cancelations} $(d)$ that $u_n=u_{x_n,x_1}=v_n$.

  Let us now show the equality $\wlslash_{\alpha,j}=\wlslash_{\beta,j}$. By Lemma~\ref{lemma:oc_crucial_variables}, $\dl_{\alpha,j}\in\{0,1,2\}$. We have therefore to consider three cases.
 \begin{enumerate}[label=\arabic*)]
\item  $\dl_{\alpha,j}=0$. In this case one deduces trivially that $\wlslash_{\alpha,j}=\epsilon=\wlslash_{\beta,j}$.
\item  $\dl_{\alpha,j}=1$. Then, by Remark~\ref{remark:cancelations} $(c)$,   $q_n=p_n=0$, $\alpha_{j-1}= \alpha'_{j-1}x_{n}^{\omega+r} u_n$, $\beta_{j-1}= \beta'_{j-1}y_{n}^{\omega+s} v_n$ for some integers $r$ and $s$,  $\wlslash_{\alpha,j}={\mathsf c}^{-1}_{x_{n},u_{n}u_0,x_{1}}\mathsf{b}_{x_{n}}^{-r'}$ and  $\wlslash_{\beta,j}={\mathsf c}^{-1}_{y_{n},v_{n}v_0,y_{1}}\mathsf{b}_{y_{n}}^{-s'}$ with $r'$ and $s'$ as above. The equality $\wlslash_{\alpha,j}=\wlslash_{\beta,j}$ is now an immediate consequence of the fact that ${\mathsf r}_{\alpha,j}\wrslash_{\alpha,j} ={\mathsf r}_{\beta,j}\wrslash_{\beta,j}$.
\item  $\dl_{\alpha,j}=2$. In this case $\dr_{\alpha,j}=0$ by Lemma~\ref{lemma:oc_crucial_variables} and one deduces from Remark~\ref{remark:cancelations} $(b)$ that $q_n=q_{n-1}=p_n=p_{n-1}=0$, $\alpha_{j-1}= \alpha'_{j-1}x_{n-1}^{\omega+r} u_{x_{n-1},x_n} x_n^\omega u_n$, $\beta_{j-1}= \beta'_{j-1}y_{n-1}^{\omega+s} u_{y_{n-1},y_n} y_n^\omega v_n$, $\wlslash_{\alpha,j}=  {\mathsf c}^{-1}_{x_{n},u_{n}u_0,x_{1}}\mathsf{b}_{x_{n}}^{-\mathbbold{q}} {\mathsf c}^{-1}_{x_{n-1},u_{x_{n-1},x_{n}},x_{n}}\mathsf{b}_{x_{n-1}}^{-r'}$ and  $\wlslash_{\beta,j}=  {\mathsf c}^{-1}_{y_{n},v_{n}v_0,y_{1}}\mathsf{b}_{y_{n}}^{-\mathbbold{q}} {\mathsf c}^{-1}_{y_{n-1},u_{y_{n-1},y_{n}},y_{n}}\mathsf{b}_{y_{n-1}}^{-s'}$.
 As above,  one deduces immediately from ${\mathsf r}_{\alpha,j}\wrslash_{\alpha,j} ={\mathsf r}_{\beta,j}\wrslash_{\beta,j}$ that ${\mathsf c}_{x_{n},u_{n}u_0,x_{1}}={\mathsf c}_{y_{n},v_{n}v_0,y_{1}}$ and $r'=s'$. So, to deduce $\wlslash_{\alpha,j}= \wlslash_{\beta,j}$ in this case, it remains to show that $x_{n-1}=y_{n-1}$. Now, ${\mathsf r}_{\alpha,j-1}$ ends with one of the variables $ \mathsf{b}_{x_{n-1}}$,  ${\mathsf c}_{\underline{\ } ,\underline{\ } ,x_{n-1}}$ and ${\mathsf i}_{\underline{\ } ,x_{n-1}}$ and, similarly,  ${\mathsf r}_{\beta,j-1}$ ends with one of the variables $\mathsf{b}_{y_{n-1}}$,  ${\mathsf c}_{\underline{\ } ,\underline{\ } ,y_{n-1}}$ and ${\mathsf i}_{\underline{\ } ,y_{n-1}}$. Since $ {\mathsf r}_{\alpha,j-1}={\mathsf r}_{\beta,j-1}$ it follows that $x_{n-1}=y_{n-1}$.
 \end{enumerate}
 We have proved that $\wlslash_{\alpha,j}=\wlslash_{\beta,j}$ in all cases. This  concludes the proof of the claim.
\end{proof}

We now show  that the number of left cancelations of occurrences of crucial variables coincides in ${\mathsf w}_{\alpha,j}$ and ${\mathsf w}_{\beta,j}$ which, in view of Claim 2, will be enough to conclude~\eqref{eq:same_cancelations}.
\begin{claim}
$\dl_{\alpha,j}=\dl_{\beta,j}$.
\end{claim}
\begin{proof}
The proof  of this claim uses induction on $j$.  By Lemma~\ref{lemma:oc_crucial_variables},  both $\dl_{\alpha,j}$ and $\dl_{\beta,j}$ belong to $\{0,1,2\}$. There are, thus, three cases to look for regarding the value of $\dl_{\beta,j}$.

\begin{enumerate}[label=\textbf{Case \arabic*.},ref=\arabic*,labelwidth=-10ex,labelsep=*,leftmargin=3ex,itemindent=7ex,parsep=1ex]
\item $\dl_{\beta,j}=0$.  By contradiction, suppose that $\dl_{\alpha,j}\neq 0$. Hence, there are two possibilities.
\begin{enumerate}[label=\textbf{Case 1.\arabic*.},ref=\arabic*,labelwidth=-10ex,labelsep=*,leftmargin=2ex,itemindent=12ex,parsep=1ex]
\item  $\dl_{\alpha,j}=2$.\label{dlalpha_2}
  Then, by Remark~\ref{remark:cancelations} $(b)$, $n>1$, $q_{n-1}=q_n=0$ and $\alpha_{j-1}= \alpha'_{j-1}x_{n-1}^{\omega+p} u_{x_{n-1},x_n} x_n^\omega u_n$. As above in the proof of Claim 1, for $p\neq 0$ this leads to a contradiction. Whence we assume that $p=0$. We  have that $\dr_{\alpha,j}=0$ by Lemma~\ref{lemma:oc_crucial_variables} and so that $\dr_{\beta,j}=0$ by Claim 1. Hence, $k=n-2$ and
\begin{alignat*}{2}
{\mathsf r}_{\alpha,j}&= {\mathsf c}^{-1}_{x_{n-2},u_{n-2},x_{n-1}  } \mathsf{b}_{x_{n-2}}^{-\mathbbold{q}_{n-2}}    {\mathsf c}^{-1}_{x_{n-3},u_{n-3},x_{n-2}}\mathsf{b}_{x_{n-3}}^{-\mathbbold{q}_{n-3}}     \cdots    {\mathsf c}^{-1}_{x_1,u_1,x_2}  \mathsf{b}_{x_{1}}^{-\mathbbold{q}_{1}}   {\mathsf c}^{-1}_{x_n,u_nu_0,x_1},\\[1mm]
{\mathsf r}_{\beta,j}&={\mathsf c}^{-1}_{y_{n-2},v_{n-2}v_0,y_1}\mathsf{b}_{y_{n-2}}^{-\mathbbold{p}_{n-2}}     {\mathsf c}^{-1}_{y_{n-3},v_{n-3},y_{n-2}}\mathsf{b}_{y_{n-3}}^{-\mathbbold{p}_{n-3}}  \cdots    {\mathsf c}^{-1}_{y_1,v_1,y_2}  \mathsf{b}_{y_{1}}^{-\mathbbold{p}_{1}}  {\mathsf c}^{-1}_{y_{n-2},v_{n-2}v_0,y_1}.
\end{alignat*}
  As ${\mathsf r}_{\alpha,j}={\mathsf r}_{\beta,j}$, we conclude that  $x_n=y_{n-2}$, $x_{n-1}=y_1$,  $u_{n-2}=v_{n-2}v_0=u_nu_0$, and, for $i\in\{1,\ldots , n-2\}$, $x_i=y_i$, $q_i=p_i$ and, when $i\neq n-2$, $u_i=v_i$.

 Furthermore,  ${\mathsf r}_{\beta,j+1}$ begins with a crucial variable of the form  ${\mathsf c}_{y_k,v_k\underline{\ } ,\underline{\ }}$ or it is equal to a terminal variable of the form ${\mathsf t}_{y_k,v_k\underline{\ } }$.
      Moreover,  either ${\mathsf r}_{\alpha,j+1}$ begins with a crucial variable of the form  ${\mathsf c}_{x_n,u_n\underline{\ } ,\underline{\ }}$, or it is equal to a terminal variable of the form ${\mathsf t}_{x_n,u_n\underline{\ } }$.
    As $u_nu_0=v_{n-2}v_0$, ${\mathsf r}_{\alpha,j+1}={\mathsf r}_{\beta,j+1}$  
    and it is not possible to make a rank 2 shift  right  at position $j$, neither in $\alpha$ nor in $\beta$,  we must have  $u_n=v_{n-2}$  and so $u_0=v_0$.
   We  have also that   either ${\mathsf r}_{\beta,j-1}$ ends with a crucial variable of the form  ${\mathsf c}_{\underline{\ },\underline{\ }v_0 ,y_1}$ or it is equal to an initial variable of the form ${\mathsf i}_{\underline{\ }v_0,y_1 }$, and that either ${\mathsf r}_{\alpha,j-1}$ ends with a crucial variable of the form  ${\mathsf c}_{\underline{\ },\underline{\ } ,x_{n-1}}$ or it is equal to an initial variable of the form ${\mathsf i}_{\underline{\ },x_{n-1} }$.  Hence,  $\alpha_j^{\omega-1}= (u_0x_1^{\omega+p_1}\cdots   u_{n-3} x_{n}^{\omega} u_nu_0 x_{1}^{\omega}u_{n-1}x_{n}^{\omega}u_n)^{\omega-1}$  and one of the  two following situations happen:
   \begin{enumerate}[label=(\roman*),ref=\roman*,labelwidth=-8ex,labelsep=*,leftmargin=5ex,itemindent=1ex,parsep=1ex]
  \item\label{primeiro}  $ \alpha_{j-1}=\alpha'_{j-1}u_0x_{1}^{\omega}u_{n-1}x_{n}^{\omega}u_n$;
  \item\label{segundo}  $ \alpha_{j-1}=u''_0x_{1}^{\omega}u_{n-1}x_{n}^{\omega}u_n$, $j>1$   and  $u'_0$ is a non-empty  suffix of $\alpha_{j-2}$ with $u_0=u'_0u''_0$.
\end{enumerate}
  If situation~(\ref{primeiro}) holds, $\alpha$ is not a canonical form since  it is possible to make a limit term shortening of type 1 and  replace $\alpha_{j-1}\alpha_j^{\omega-1}$ by $\alpha'_{j-1}(u_0x_1^{\omega+p_1}\cdots   u_{n-3} x_{n}^{\omega} u_n)^{\omega-1}$.  In particular,  this proves already the impossibility of Case~1.1. for $j=1$.

 Suppose now that situation~(\ref{segundo}) holds. Then $j>1$ and we will use the induction hypothesis to obtain a contradiction.
 Note that $x_n^\omega u_nu'_0 $  can not be the final $\omega$-portion of $\alpha_{j-2}$ (otherwise it would be possible to make an agglutination). Consequently,  $\dr_{\beta,j-2}=\dr_{\alpha,j-2}=0$   and  $\crl{{\mathsf r}_{\beta,j-1}}=\crl{{\mathsf r}_{\alpha,j-1}}=1$.
   Furthermore  ${\mathsf r}_{\beta,j-1}={\mathsf r}_{\alpha,j-1}={\mathsf c}_{z, wu_0,x_1}$ where $z^\omega wu'_0$ is the final $\omega$-portion of $\alpha_{j-2}$. Hence, the final $\omega$-portion of $\beta_{j-2}$ is $z^\omega w'$ with $w'$ a prefix of $w$. Assuming  by induction hypothesis that  $\dl_{\beta,j-2}=\dl_{\alpha,j-2}$, we have from Claim~2 that  $\alpha_{j-2}=\beta_{j-2}$, and  one deduces that $w=w'$ and $u'_0=\epsilon$. So, actually, situation~(\ref{segundo}) can not happen either.

 \item \label{dlalpha_1} $\dl_{\alpha,j}=1$. So $\alpha_{j-1}$ is of the form $\alpha_{j-1}=\alpha'_{j-1}x_{n}^{\omega+p}u_n$, $k=n-1$ and $q_n=0$.
   If  $p\neq 0$, then we get a contradiction analogously. So, we assume additionally that  $p=0$.
       Thereby, we get
\begin{alignat*}{2}
{\mathsf r}_{\alpha,j}&= {\mathsf c}^{-1}_{x_{n-1},u_{n-1},x_{n}}\mathsf{b}_{x_{n-1}}^{-\mathbbold{q}_{n-1}}    {\mathsf c}^{-1}_{x_{n-2},u_{n-2},x_{n-1}}\mathsf{b}_{x_{n-2}}^{-\mathbbold{q}_{n-2}}     \cdots    {\mathsf c}^{-1}_{x_1,u_1,x_2}{\mathsf r}'_{\alpha,j},\\[1mm]
{\mathsf r}_{\beta,j}&= {\mathsf c}^{-1}_{y_{n-1},v_{n-1}v_0,y_{1}}\mathsf{b}_{y_{n-1}}^{-\mathbbold{p}_{n-1}}   {\mathsf c}^{-1}_{y_{n-2},v_{n-2},y_{n-1}}\mathsf{b}_{y_{n-2}}^{-\mathbbold{p}_{n-1}}\cdots    {\mathsf c}^{-1}_{y_1,v_1,y_2}{\mathsf r}'_{\beta,j},
\end{alignat*}
for some words ${\mathsf r}'_{\alpha,j}, {\mathsf r}'_{\beta,j}\in(\mathsf{V}^{-1})^*$.
As ${\mathsf r}_{\alpha,j}={\mathsf r}_{\beta,j}$, we conclude that ${\mathsf r}'_{\alpha,j}= {\mathsf r}'_{\beta,j}$, $u_{n-1}=v_{n-1}v_0$, $x_n=y_1$,  for $i\in\{1,\ldots , n-1\}$, $x_i=y_i$   and $p_i=q_i$ if $i\ge 2$, and  $u_i=v_i$ when $i\neq n-1$. Whence
\begin{alignat*}{2}
\alpha_{j-1}\alpha_j^{\omega-1}=&\ \alpha'_{j-1}x_n^\omega u_n(u_{0}x_{n}^{\omega+q_{1}} u_{1}x_2^{\omega+q_2}u_2 \cdots
x_{n-1}^{\omega+q_{n-1}}u_{n-1}x_n^\omega u_n)^{\omega-1}, \\
\beta_{j-1}\beta_j^{\omega-1}=&\ \beta_{j-1}(v_{0}x_{n}^{\omega+p_{1}} u_{1}x_2^{\omega+q_2}u_2 \cdots
x_{n-1}^{\omega+q_{n-1}}v_{n-1})^{\omega-1}.
\end{alignat*}

 Suppose now that $\dr_{\alpha,j}=1$. In this case $\dr_{\beta,j}=1$,  $q_1=p_1=0$ and  $u_0=v_0=\epsilon$. So, $u_{n-1}=v_{n-1}$, it do not occur crucial variables neither in ${\mathsf r}'_{\alpha,j}$ nor in $ {\mathsf r}'_{\beta,j}$  and $\alpha_{j-1}\alpha_j^{\omega-1}\alpha_{j+1}$ is of the form
$$\alpha_{j-1}\alpha_j^{\omega-1}\alpha_{j+1}=\alpha'_{j-1}x_{n}^{\omega}u_n(x_{n}^{\omega} u_{1}\cdots
x_{n-1}^{\omega+q_{n-1}}u_{n-1}x_{n}^{\omega}u_n)^{\omega-1}x_{n}^{\omega+r}\alpha'_{j+1}$$ with $r\neq 0$. Therefore, $\alpha$ is not a canonical form since  by the application of a limit term shortening of type 2, it should be replaced  by
$\alpha'_{j-1}(x_{n}^{\omega} u_{1}\cdots
x_{n-1}^{\omega+q_{n-1}}u_{n-1})^{\omega-1}x_{n}^{\omega+r}\alpha'_{j+1}$.

Suppose next that $\dr_{\alpha,j}=0$ and so that $\dr_{\beta,j}=0$. Then ${\mathsf r}'_{\alpha,j}  =
\mathsf{b}_{x_{n}}^{-\mathbbold{q}_{1}}  {\mathsf c}_{x_{n}, u_{n}u_0,x_{n}}^{-1}$ and ${\mathsf r}'_{\beta,j}=
\mathsf{b}_{x_{n}}^{-\mathbbold{p}_{1}} {\mathsf c}_{x_{n-1}, v_{n-1}v_0,x_{n}}^{-1}$. Therefore  $x_n=x_{n-1}$,
$q_1=p_1$ and $u_nu_0=v_{n-1}v_0\,  (=u_{n-1})$. As in the previous case, analysing the first crucial variable
of the remainder at position $j+1$ and the last one of the  remainder at position $j-1$, we conclude that it
must be $u_n=v_{n-1}$, and so that $u_0=v_0$,   and $u_0 x_{n}^{\omega}u_n $ is a suffix of
$\alpha_{j-2}\alpha_{j-1}$.  Consequently,  one of the   two following situations happen:
   \begin{enumerate}[label=(\roman*),ref=\roman*,labelwidth=-8ex,labelsep=*,leftmargin=5ex,itemindent=1ex,parsep=1ex]
  \item\label{pprimeiro}  $ \alpha_{j-1}=\alpha'_{j-1}u_0 x_{n}^{\omega}u_n $;
  \item\label{ssegundo}  $\alpha_{j-1}=u''_0 x_{n}^{\omega}u_n$, $j>1$   and  $u'_0$ is a non-empty  suffix of $\alpha_{j-2}$ where $u_0=u'_0u''_0$.
\end{enumerate}

If~(\ref{pprimeiro}) holds, then we can apply a type  1 limit term  shortening in
$\alpha_{j-1}\alpha_j^{\omega-1}$ and replace it by $\alpha'_{j-1}(u_{0}x_{n}^{\omega+q_{1}} u_{1}
x_{2}^{\omega+q_{2}}\cdots x_{n}^{\omega+q_{n-1}}u_{n})^{\omega-1}$. Thus $\alpha$ is not a canonical form.
  Notice that, if $j=1$,  this proves the impossibility of Case~1.2.  If $j>1$, it remains to consider situation~(\ref{ssegundo}), in which case $\dl_{\beta,j-1}=\dl_{\alpha,j-1}=0$ and  $\crl{{\mathsf r}_{\beta,j-1}}=\crl{{\mathsf r}_{\alpha,j-1}}=1$.
   Furthermore  ${\mathsf r}_{\beta,j-1}={\mathsf r}_{\alpha,j-1}={\mathsf c}_{z, wu_0,x_n}$ where $z^\omega wu'_0$ is the final $\omega$-portion of $\alpha_{j-2}$. Consequently, the final $\omega$-portion of $\beta_{j-2}$ is $z^\omega w'$ with $w'$ a prefix of $w$. Again assuming  by induction hypothesis that  $\dl_{\beta,j-2}=\dl_{\alpha,j-2}$, we have that  $\alpha_{j-2}=\beta_{j-2}$, and this implies that $w=w'$ and $u'_0=\epsilon$. So, situation~(\ref{ssegundo}) does not actually occur either.
\end{enumerate}
In both situations we reached a contradiction. Therefore $\dl_{\alpha,j}=0$ when $\dl_{\beta,j}=0$. By symmetry
it follows that $\dl_{\alpha,j}=0$ if and only if $\dl_{\beta,j}=0$.

\item  $\dl_{\beta,j}=2$. Then $\dr_{\alpha,j}=\dr_{\beta,j}=0$,
    and $\dl_{\alpha,j}\neq 0$ by Case 1. Suppose that $\dl_{\alpha,j}=1$. Hence  $k=n+1$, $q_n=p_n=p_{n+1}=0$, and
$\alpha_{j-1}$  and $\beta_{j-1}$ are of the forms, respectively,
$\alpha_{j-1}=\alpha'_{j-1}x_{n}^{\omega+q}u_n$ and
$\beta_{j-1}= \beta'_{j-1}y_{n}^{\omega+p}v_{n}y_{n+1}^{\omega}v_{n+1}$.
Furthermore,
\begin{alignat*}{2}
{\mathsf r}_{\alpha,j}& =\mathsf{b}_{x_{n}}^{q'}\mathsf{c}_{x_{n-1}, u_{n-1},x_{n}}^{-1}\cdots \mathsf{b}_{x_{2}}^{-\mathbbold q_{2}}\mathsf{c}_{x_{1}, u_{1},x_{2}}^{-1} \mathsf{b}_{x_{1}}^{-\mathbbold q_{1}}\mathsf{c}_{x_{n}, u_{n}u_{0},x_{1}}^{-1},\\[1mm]
{\mathsf r}_{\beta,j}& =\mathsf{b}_{y_{n}}^{p'}\mathsf{c}_{y_{n-1}, v_{n-1},y_{n}}^{-1}\cdots \mathsf{b}_{y_{2}}^{-\mathbbold p_{2}}\mathsf{c}_{y_{1}, v_{1},y_{2}}^{-1} \mathsf{b}_{y_{1}}^{-\mathbbold p_{1}}\mathsf{c}_{y_{n+1}, v_{n+1}v_{0},y_{1}}^{-1},
\end{alignat*}
where, for $t\in\{p,q\}$, $t'$ is $0$ when $t\geq 0$ and it is $t$ when $t< 0$. From the equality ${\mathsf
r}_{\alpha,j}={\mathsf r}_{\beta,j}$  it then follows that $q'=p'$,   $x_n=y_n=y_{n+1}$, $u_nu_0=v_{n+1}v_0$
and, for $i\in\{1, \cdots, n-1\}$,   $ x_i=y_i$, $u_i=v_i$ and $p_i=q_i$.  Again, analysing the first crucial
variables of ${\mathsf r}_{\alpha,j+1}$ and $ {\mathsf r}_{\beta,j+1}$, we conclude that $u_n=v_{n+1}$, so that
$u_0=v_0$. Whence,
\begin{alignat*}{2}
\beta_{j-1}\beta_j^{\omega-1}=\beta'_{j-1}x_{n}^{\omega+p}v_{n}x_{n}^{\omega}u_n
(u_{0}x_{1}^{\omega+q_{1}}u_1\cdots  x_{n-1}^{\omega+ q_{n-1}}u_{n-1} x_{n}^{\omega}v_{n}x_{n}^{\omega}u_n)^{\omega-1}.
\end{alignat*}
So, $\beta$ is not a canonical term, as it allows the application of a limit term shortening of type 2 in which
$\beta_{j-1}\beta_j^{\omega-1}$ should be replaced by $\beta'_{j-1}x_{n}^{\omega+p}u_n
(u_{0}x_{1}^{\omega+q_{1}}u_1\cdots  x_{n-1}^{\omega+ q_{n-1}}u_{n-1} x_{n}^{\omega}u_n)^{\omega-1}$. This is in
contradiction with the hypothesis and so $\dl_{\alpha,j}=2=\dl_{\beta,j}$.

\item $\dl_{\beta,j}=1$. From the previous cases it is now immediate that $\dl_{\alpha,j}=\dl_{\beta,j}=1$.
\end{enumerate}

We have proved in all cases that $\dl_{\alpha,j}=\dl_{\beta,j}$ and, so,  the proof of the claim is complete.
\end{proof}

The ending of the proof of the proposition is now clear. By Claim 3, $\dl_{\alpha,j}=\dl_{\beta,j}$ and, so, by Claim 2 (which uses Claim 1)  one deduces that $\wlslash_{\alpha,j}=\wlslash_{\beta,j}$ and $\wrslash_{\alpha,j}= \wrslash_{\beta,j}$ for every odd position $j$. As observed above this entails that
 ${\mathsf w}_{\mathbbold{q}}(\alpha) = {\mathsf w}_{\mathbbold{q}}(\beta)$ and, so, as $\alpha$ and $\beta$ are canonical forms, that $\alpha=\beta$.
\end{proof}

\section{Main results}\label{section:main results}
 The following fundamental theorem is an immediate consequence of Propositions~\ref{prop:nec_condit_eq_kterms} and~\ref{prop:suff_condit_eq_kterms} and of Corollary~\ref{cor:2-length_of_cf}.
\begin{theorem}\label{theo:equal_canonical_forms}
Let  $\alpha$ and $\beta$ be canonical $\kb$-terms. If  $\LG\models \alpha=\beta$, then $\alpha=\beta$.
\end{theorem}

The main results of this paper may now be easily deduced.

\begin{theorem}\label{theorem:LGkappa_word_problem_dec}
 The $\kb$-word problem for  ${\bf LG}$ is decidable.
\end{theorem}
\begin{proof} The solution of the $\kb$-word problem for  ${\bf LG}$ consists in, given two \kbt s $\alpha$ and
$\beta$, to compute their respective canonical forms $\alpha'$ and $\beta'$.  Then  $\LG\models
\alpha=\beta$ if and only if  $\alpha'=\beta'$. Alternatively, in case $\alpha$ and $\beta$ are not rank 0,  to verify whether $\LG\models
\alpha=\beta$ one can compute rank 1 canonical forms or rank 2 semi-canonical forms $\alpha''$ and $\beta''$, equivalent to $\alpha$ and $\beta$, and verify whether  $\widetilde{\mathsf w}_\mathbbold{q}(\alpha)=\widetilde{\mathsf w}_\mathbbold{q}(\beta)$ for $\mathbbold{q}={\rm max}\{\mathbbold{q}_{\alpha},\mathbbold{q}_{\beta}\}$.
\end{proof}

\begin{theorem}\label{theorem:kappa_basis_LG}
 The set $\Sigma$ is a basis of $\kb$-identities for $\LG^{\kb}$.
\end{theorem}
\begin{proof} We have to prove that, for all \kbt s $\alpha$ and $\beta$,
$\LG\models \alpha=\beta$ if and only if $\Sigma \vdash \alpha=\beta$. The only if part follows from the fact
that $\LG$ verifies all the $\kb$-identities of $\Sigma$. For the if part recall that, by
Proposition~\ref{prop:rank2} and by the rank 1 and rank 2 canonical form algorithms, there exist canonical forms
$\alpha'$ and $\beta'$ that may be computed from  $\alpha$ and $\beta$ using the
$\kb$-identities of $\Sigma$. Therefore, if $\LG\models \alpha=\beta$ then $\LG\models \alpha'=\beta'$ and
so, by Theorem~\ref{theo:equal_canonical_forms}, $\alpha'=\beta'$. Since $\Sigma \vdash \alpha=\alpha'$ and
$\Sigma \vdash \beta=\beta'$ it follows by transitivity that $\Sigma \vdash \alpha=\beta$.
\end{proof}
\section*{Acknowledgments}
This work  was supported by the European Regional Development Fund, through the programme COMPETE, and by the
Portuguese Government through FCT -- \emph{Funda\c c\~ao para a Ci\^encia e a Tecnologia}, under the project
 PEst-C/MAT/UI0013/2014.


\end{document}